\documentclass[12pt, reqno]{amsart}
\usepackage{amsmath,amssymb,amsthm}
\usepackage{amsmath, amssymb}
\usepackage{amsthm, amsfonts, mathrsfs}
\usepackage{mathptmx}
\usepackage{fullpage}
\usepackage{amsfonts,graphicx}
\numberwithin{equation}{section}
\usepackage[colorlinks=true, pdfstartview=FitV, linkcolor=blue, citecolor=blue, urlcolor=blue]{hyperref}

\numberwithin{equation}{section}
\newtheorem{theorem}{Theorem}[section]
\newtheorem{proposition}[theorem]{Proposition}
\newtheorem{lemma}[theorem]{Lemma}

\numberwithin{equation}{section}

\theoremstyle{remark}
\newtheorem{remark}[theorem]{Remark}
\def\T{ \mathbb{T} }

\newcommand{\p}{{\mathbb{P}}}
\newcommand{\R}{{\mathbb{R}}}
\def\div{ \hbox{\rm div}\,  }
\def\u{ \mathbf{u} }
\def\G{{g} }
\def\n{ \mathbf{n} }
\def\b{ \mathbf{B} }
\def\h{ \mathbf{b} }
\def\T{ \mathbb{T} }
\newcommand{\Z}{{\mathbb{Z}}}

\def\la{ \Lambda }
\def\ta{ \theta }
\def\vta{ \vartheta }
\def\nn{\nonumber}

\newcommand{\beq}{\begin{equation}}
\newcommand{\eeq}{\end{equation}}
\newcommand{\norm}[2]{\lVert #1 \rVert_{#2}}
\newcommand{\norma}[2]{\left\lVert #1 \right\rVert_{#2}}

\begin{document}
\title[3D Inviscid Non-Isentropic MHD]{Magnetic Stabilization of Compressible Flows: Global Existence in 3D Inviscid Non-Isentropic MHD Equations}

\author[Wu, Xu and Zhai]{Jiahong Wu$^{1}$, Fuyi Xu$^{2}$ and Xiaoping Zhai$^{3}$}

\address{$^1$ Department of Mathematics, University of Notre
	Dame, Notre Dame, IN 46556, USA}

\email{jwu29@nd.edu}

\address{$^{2}$School of Mathematics and  Statistics, Shandong University of
	Technology,  Zibo 255049,  Shandong Province,    P.R. China}

\email{zbxufuyi@163.com}

\address{$^{3}$   School of Mathematics and Statistics, Guangdong University of
	Technology,  Guangzhou 510520, P.R. China}

\email{pingxiaozhai@163.com (Corresponding author)}



\begin{abstract}
Solutions to the compressible Euler equations in all dimensions have been shown
to develop finite-time singularities from smooth initial data such as shocks
and cusps. There is an extraordinary list of results on this subject. When the
inviscid compressible flow is coupled with the magnetic field in the 3D
inviscid non-isentropic compressible magnetohydrodynamic (MHD) equations in
$\mathbb{T}^3$, this paper rules out finite-time blowup and establishes the global
existence of smooth and stable solutions near a suitable background magnetic
field. This result rigorously confirms the stabilizing phenomenon observed in
physical experiments involving electrically conducting fluids.
\end{abstract}
\maketitle


\section{ Introduction and Main Result}

This paper aims to rigorously investigate the stabilizing phenomenon observed in physical experiments, using the example of the 3D inviscid, heat-conductive, compressible magnetohydrodynamic (MHD) equations near a background magnetic field. The 3D non-isentropic compressible MHD system assumes the form
\begin{eqnarray}\label{m0}
	\left\{\begin{aligned}
		&\partial_t \rho + \div (\rho \u) =0, \quad t>0, \,\, x\in \mathbb
		T^3,\\
		&\rho\partial_t \u+ \rho\u\cdot\nabla \u+\nabla
		P=(\nabla\times\h)\times\h,\\
		& c_\nu(\rho \vta_t +  \rho \u \cdot \nabla \vta)-\kappa\Delta\vta +P
		\div \u  = \sigma|\nabla\times\h|^2,\\
		&\partial_t
		\h-\sigma\Delta\h+\u\cdot\nabla\h-\h\cdot\nabla\u+\h\div\u=0,\\
		&\div \h =0,
	\end{aligned}\right.
\end{eqnarray}
where $\mathbb T^3$ is the 3D periodic box, and $\rho =\rho(t,x)$,  $\u=
\u(t,x)$, $\vta=\vta(t,x)$ and  $\h= \h(t,x)$
denote the density, the velocity field, the temperature and the magnetic
field, respectively. The positive parameters $c_\nu$, $\kappa$ and $\sigma$ are
the specific heat at constant volume, the coefficient of heat conduction  and
the magnetic diffusivity, respectively.
The pressure $P=P(\rho, \vta)$ is assumed to be of the form
\beq\label{pres}
P(\rho, \vta) =R \rho \vta
\eeq
with a universal constant $R>0$. We remark that the main result
presented in this
paper actually hold for the following more general pressure laws $P(\rho, \vta)
= \pi_0(\rho) + \vta \pi_1(\rho)$ when the smooth functions $\pi_0$ and $\pi_1$
satisfy some very general constraints.

\vskip .1in
The compressible MHD models considered here provide the principal framework for
the theoretical
description of turbulence in the solar wind. Since the observed fluctuations
involve density variations, the effects of plasma compressibility should be
incorporated in the theory \cite{Gold}.

\vskip .1in
The motivation for studying the global existence, stability and large-time
behavior of \eqref{m0} comes from two distinct
sources. The first is the stabilizing phenomenon observed in physical
experiments. An important issue in the MHD turbulence theory is to understand
the influence of the magnetic field on bulk MHD turbulence. Various experiments
on electrically conducting fluids such as liquid metals have observed that
the background magnetic fields can actually stabilize these MHD flows (see,
e.g., \cite{AMS, Alex, Alf,  Davi0, Davi1, DA, Gall, Gall2,jiangfei2023, Mof}).
Our intention has been to understand the mechanism and establish this
phenomenon as mathematically rigorous facts. The second motivation is
mathematical. Solutions of the compressible Euler equations with the ideal gas
law in all dimensions (1D, 2D and 3D) have been shown to form finite-time
singularities from smooth initial data such as shocks and cusps, due to an
outstanding list of research works on this subject (see, e.g., \cite{Buc1,
Buc2, Buc3, Chr1, Chr2, Luk, Mer, Sid, Xin, Yin}).  Our intention here is to
provide a
global smooth and stability result for the compressible MHD system when the
compressible Euler is coupled with the magnetic field near a suitable
background.

\vskip .1in
The background magnetic field $\n\in\R^3$ is assumed to satisfy the following
Diophantine condition, for any $\mathbf{k}\in\Z^3\setminus \{\mathbf{0}\},$
\begin{equation}\label{diufantu}
	|\n\cdot \mathbf{k}|\ge \frac{c}{|\mathbf{k}|^r} \quad\mbox{for some $c>0$
	and $r>2$.}
\end{equation}

We remark that studying the dynamics near a vector field satisfying the
Diophantine condition has been a common practice in ergodic theory and
dynamical systems (see, e.g., \cite{Cas, Koc, Lop}). As shown by Chen,  Zhang and Zhou
\cite{zhangzhifei}, almost all vector fields
in $\R^3$ satisfy \eqref{diufantu}. Of course, there are vectors that do not
satisfy the Diophantine condition such as those with all three components
being rational. A crucial fact about a vector field $\n\in\R^3$ satisfying the
Diophantine condition is the following Sobolev inequality
for any function $f$
satisfies  $\nabla f\in H^{s+r}(\T^3)$ and $\int_{\T^3}f\,dx=0$,
\begin{equation}
	\|f\|_{H^{s}(\T^3)}\le C\|{\n}\cdot\nabla f\|_{H^{s+r}(\T^3)}. \label{Sob}
\end{equation}

\vskip .1in
We remark that there is a very large literature on the stability problem
concerning the incompressible MHD equations near a background magnetic field.
Studies on the compressible MHD stability problem is
relatively more recent and important progress has been made (see, e.g.,
\cite{DongWuZhai, Huang1, jiangfei2019, JZW, Hong, HuWang, LS2D, Sun, TanWang, WangXin,
WuWu, WuZhu}). Wu and Wu \cite{WuWu} systematically investigated the stability
problem on the 2D compressible MHD equations with velocity dissipation but
without magnetic diffusion near a background magnetic field. The spatial domain
is the whole space $\mathbb R^2$. A key discovery of this paper is that the
system governing the perturbations can be converted into fourth-order wave
equations. In contrast, for the incompressible MHD flows, the wave equations
are in general second-order.  The corresponding stability problem for the 3D
compressible MHD with velocity dissipation and no magnetic diffusion in
$\mathbb R^3$ remains open. When the spatial domain is the 2D periodic domain
$\mathbb T^2$, Wu and Zhu \cite{WuZhu} solved the stability problem on the 2D
non-resistive MHD equation by constructing the equations of combined
quantities and making use of the wave structures. In the corresponding 3D
periodic case, Wu and Zhai \cite{WuZhai} solved the MHD stability problem with
velocity dissipation near a background magnetic field $\n\in\R^3$ satisfying
Diophantine condition. When the fluid is governed by the inviscid compressible
Euler equations, the situation becomes much more difficult and the goal of this
paper is to give a definite answer to this challenging open problem. It is worth noting the beautiful work of Wang and Xin \cite{WangXin0}, which investigates the well-posedness of the inviscid, heat-conductive, and resistive compressible MHD system in a horizontally periodic flat strip domain. However, they approach does not apply to the periodic setting of this paper.

\vskip .1in
For any positive constants $\bar \rho$ and $\bar \vta$,  it is easy to verify
that
$(\bar \rho, {0}, \bar \vta, \n)$ is an equilibrium state solution of
(\ref{m0}).
Without loss of generality, we take $\bar \rho = \bar \vta =1$. The
perturbation $(a, \u, \theta, \b)$ with
$$
a = \rho -1, \quad \theta = \vta -1 \quad\mbox{and}\quad \b = \h -\n
$$
satisfies the following MHD system
\begin{eqnarray}\label{m1}
	\left\{\begin{aligned}
		&\partial_t a + \div ((1+a)\u) =0,\\
		&(1+a)\partial_t \u+ (1+a)\u\cdot\nabla \u+\nabla
		P=\n\cdot \nabla \b -\nabla ({\n}\cdot\b) +\b\cdot\nabla\b-\b\nabla\b
		,\\
		& c_\nu((1+a) \partial_t\theta +   (1+a)\u \cdot \nabla
		\theta)-\kappa\Delta\theta +P
		\div \u  = \sigma|\nabla\times\b|^2,\\
		&\partial_t
		\b-\sigma\Delta\b+\u\cdot\nabla\b-\b\cdot\nabla\u+\b\div\u=\n\cdot\nabla
		 \u-\n\div\u,\\
		&\div \b =0.
	\end{aligned}\right.
\end{eqnarray}
For simplicity, we set the parameter $c_\nu=1$. Denoting
$$
\bar{\kappa}(a)\stackrel{\mathrm{def}}{=}\frac{\kappa}{1+a},
\quad I(a)\stackrel{\mathrm{def}}{=}\frac{a}{1+a}\quad\hbox{and}\quad
J(a)=\ln({1+a} ),
$$
separating the linear parts from the nonlinear ones in (\ref{m1}) and using
(\ref{pres}), we have
\begin{eqnarray}\label{m2}
	\left\{\begin{aligned}
		&\partial_t a + \div\u  =f_1,\\
		&\partial_t \u+ R\nabla a+ R\nabla \ta
		={\n}\cdot\nabla \b-\nabla ({\n}\cdot\b)+f_2,\\
		&\partial_t\theta-\kappa\Delta\ta+ \div \u=f_3,\\
		&\partial_t \b-\sigma\Delta\b={\n}\cdot\nabla \u-\n\div\u+f_4,\\
		&\div \b =0,\\
		&(a,\u,\ta,\b)|_{t=0}=(a_0,\u_0,\ta_0,\b_0),
	\end{aligned}\right.
\end{eqnarray}
where
\begin{align}
	f_1\stackrel{\mathrm{def}}{=}&-\u\cdot\nabla a-a\div \u,\nn\\
	f_2\stackrel{\mathrm{def}}{=}&-\u\cdot\nabla
	\u+\b\cdot\nabla\b-\b\nabla\b+ R I(a)\nabla a- R \theta\nabla J(a)
	\nn\\
	&-I(a)({\n}\cdot\nabla \b+\b\cdot\nabla\b-{\n}\nabla \b-\b\nabla\b),\nn\\
	f_3\stackrel{\mathrm{def}}{=}&- \div(\theta \u)-\kappa
	I(a)\Delta\ta+\frac{|\nabla\times\b|^2}{1+a},\nn\\
	f_4\stackrel{\mathrm{def}}{=}&-\u\cdot\nabla\b+\b\cdot\nabla\u-\b\div\u.\nn
\end{align}

We make the following minor assumptions on the initial data,
\begin{align}
&\frac1{|\mathbb T^3|} \int_{\mathbb T^3} \rho_0(x) \,dx =1, \quad
\int_{\mathbb T^3} \rho_0(x) \u_0(x) \,dx =\int_{\mathbb T^3} \b_0(x) \,dx =0,\label{conser1}\\
&\int_{\mathbb T^3} \rho_0 \ta_0 \,dx+\frac12\int_{\mathbb T^3} \rho_0
	|\u_0|^2dx+\frac12\int_{\mathbb T^3} |\b_0|^2\,dx=0\label{conser1+1}.
\end{align}
The properties in \eqref{conser1} and \eqref{conser1+1} are preserved in time.
As a consequence, we
are able to apply Poincar\'e's inequality on $a$ and $\b$.
Poincar\'e type inequalities can also be established for $\u$ and
$\ta$, but they require more elaborated proofs due to the lack of the mean-zero
condition on $\u$ or $\ta$. A Poincar\'e type inequality is shown in Section
\ref{sec:l2} while a generalized Poincar\'e type inequality for $\ta$ is
provided in Section \ref{sec: tap}.
 Under these
minor assumptions, we are able to show that the MHD system governing the
perturbations \eqref{m2} always has a unique global smooth solution if the
initial data are sufficiently small. In addition, the perturbation is
asymptotically stable and decays to the equilibrium state solution
algebraically in time. More precisely, we establish the following theorem.

\begin{theorem}\label{dingli}
	For any ${N}\ge 4r+7$ with $r>2$. Assume that the initial data $(\rho_0,
	\u_0, \ta_0, \b_0)$ satisfies  \eqref{conser1}, \eqref{conser1+1}
	and, for $a_0= \rho_0-1$ and $\theta_0 =\vta_0-1$,
	\begin{align*}
		(a_0, \theta_0)\in H^{N}(\T^3),\quad
		c_0\le\rho_0, \ta_0\le c_0^{-1},\quad(\u_0, \b_0)\in H^{N}(\T^3)
	\end{align*}
	for some constant $c_0>0$.
Then there exists a small constant $\varepsilon>0$ such that, if
	\begin{align*}
		\norm{a_0}{H^{N}}+ \norm{\u_0}{H^{N}} +\norm{\theta_0}{H^{N}}+\norm{
		\b_0}{H^{N}}\le\varepsilon,
	\end{align*}
	then the system \eqref{m2} admits a  unique global  solution
	$(a,
	\u, \theta,\b)\in C([0,\infty );H^{N})$. Moreover, for any $t\ge
	0$ and $r+4\le\beta <N $,
	there holds
	\begin{align*}
		\norm{a(t)}{H^{\beta}}+\norm{\u(t)}{H^{\beta}}+\norm{\theta(t)}{H^{\beta}}+\norm{\b(t)}{H^{\beta}}\le
		 C(1+t)^{-\frac{3({N}-\beta)}{2({N}-r-4)}}.
	\end{align*}
\end{theorem}

This result rigorously confirms the stabilizing phenomenon observed in physical experiments involving electrically conducting fluids. The stability result, along with its proof, elucidates the mechanism by which the magnetic field exerts a stabilizing effect on compressible MHD flows. It provides an important example of how magnetic fields can suppress instabilities in inviscid fluid dynamics. We also highlight the significant stabilization results on inviscid flows obtained in the influential works \cite{guoyan1} and \cite{guoyan2}.

\begin{remark}
	It is not clear whether Theorem \ref{dingli} can be extended to the 3D
	inviscid
	isentropic compressible MHD equations. As we shall see in the proof of
	Theorem \ref{dingli}, we need an enhanced dissipation property on the
	quantity $\div \u$. This property is obtained by combining the equations of
	$\div \u$ and of $\theta$. Without the equation of $\theta$, it is not
	clear
	how to gain this extra regularity on $\div \u$.
\end{remark}

\begin{remark}
	Even though the focus of this paper is on the 3D case, a similar result on
	the corresponding 2D compressible MHD equations near a background
	satisfying the Diophantine condition can be established. Furthermore, in
	the 2D case, we can also prove the desired stability near a background
	magnetic field that is not even Diophantine if the initial perturbations
	obey some symmetry conditions.
\end{remark}

There are major difficulties in proving Theorem \ref{dingli}. When the magnetic
field is not present, the inviscid incompressible flow is governed by the
compressible Euler equations. As aforementioned, compressible Euler equations
develop finite-time singularities even when the initial data is smooth and
small. This makes the MHD stability problem appear impossible. The only hope is
that the magnetic field can smooth and stabilize the fluid.

\vskip .1in
This paper develops a very effective approach to maximally exploit the
smoothing and stabilizing effect due to coupling and interaction. The equation
for the perturbation of the density $a$ is given by
$$
\partial_t a + \div\u  =f_1,
$$
which involves no damping or dissipation. However, when it is coupled with
$\div\u$, their interaction generates a wave structure.  For simplicity, we
explain this stabilizing mechanism in terms of the
linearized equations
of $a$ and $\div \u$, which are given by
\begin{align*}
	& \partial_t a + \div \u =0,\\
	& \partial_t \div \u+ R\Delta a+ R\Delta \ta
	=-\Delta ({\n}\cdot\b).
\end{align*}
We can easily converted this system into the following wave equations
\begin{align*}
	& \partial_{tt} a - R \Delta a = R \Delta \theta - \Delta({\n}\cdot\b) ,\\
	& \partial_{tt} \div \u - R\Delta\div \u =- R\Delta \partial_t\ta
	-\Delta ({\n}\cdot\partial_t\b).
\end{align*}
Making use of this structure by constructing suitable Lyapunov functional, we
are able to obtain the dissipative effect of $a$. In fact, under the
bootstrapping argument assumption that
$$
\norm{(a(t), \u(t), \ta(t),
\b(t))}{H^{N}}\le \delta,
$$
we obtain
\begin{align}
	&\norm{\nabla a}{H^{{r+3}}}^2+\sum_{0\le s\le
	r+3}\frac{d}{dt}{\big\langle{\Lambda^{s}} \u,{\Lambda^{s}}(\nabla
	a)\big\rangle} \nn\\
	&\quad\le C\norm{\div\u}{H^{{r+3}}}^2+
	C(1+\delta^2)\norm{ (\ta,\b)}{H^{r+4}}^2+ C\delta^2 \norm{{\n}\cdot\nabla
	\u}{H^{r+3}}^2, \label{tt}
\end{align}
which allows us to obtain the time integrability of $\norm{\nabla
a}{H^{{r+3}}}^2$. However, this process also generates two bad terms,
$\norm{\div\u}{H^{{r+3}}}^2$ and $\norm{{\n}\cdot\nabla
\u}{H^{r+3}}^2$. We need to obtain their time integrability in order to bound
the time integral of $\norm{\nabla
a}{H^{{r+3}}}^2$.

\vskip .1in
Due to the lack of damping and dissipation in the equation of $\u$, the time
integrability of $\norm{{\n}\cdot\nabla \u}{H^{r+3}}^2$ appears impossible if
we follow classical approaches. We are able to discover the mathematical
mechanism behind the stabilizing phenomenon observed in physical experiments.
Mathematically the interaction of the fluid and the magnetic field
near a background magnetic field generates a wave structure. For the sake of
simplicity, we consider the linearized system of $\u$ and $\b$,
\begin{align*}
	&\partial_t \u+ R\nabla a+ R\nabla \ta
	={\n}\cdot\nabla \b-\nabla ({\n}\cdot\b),\\
	&\partial_t \b-\sigma\Delta\b={\n}\cdot\nabla \u-\n\div\u.
\end{align*}
After ignoring the irrelevant terms $R\nabla a+ R\nabla \ta$, we obtain
the following degenerate wave equations
\begin{align*}
	&		\partial_{tt} \u - \sigma \Delta \partial_t \u -
	({\n}\cdot\nabla)^2 \u
	=-\nabla((\n\otimes \n)\cdot \nabla \u) +\nabla\div \u- (\n\cdot
	\nabla\div\u)\n, \\
	&	\partial_{tt} \b - \sigma \Delta \partial_t \b - ({\n}\cdot\nabla)^2
	\b
	=-\nabla((\n\otimes \n)\cdot \nabla \b) +\n \Delta
	({\n}\cdot\b).
\end{align*}
$\u$ and $\b$ share a very similar wave structure. In comparison with the
original equation of $\u$, the wave equation contains two extra regularizing
terms. $- \sigma \Delta \p_t \u$ comes from the magnetic diffusion and $-
({\n}\cdot\nabla)^2 \u$ is due to the magnetic field. $- ({\n}\cdot\nabla)^2
\u$ allows us to control the directional derivative of $\u$ along the
background magnetic
field. This reflects the observed stabilizing effect of fluids in
the direction of the background magnetic field. Making use of this special wave
structure, we are able to establish the following estimate
\begin{align}
	&\norm{{\n}\cdot\nabla \u}{H^{r+3}}^2-\sum_{0\le s\le
		r+3}\frac{d}{dt}\big\langle{\Lambda^s}\b,{\Lambda^s}({\n}\cdot\nabla
	\u)\big\rangle\le C\norm{ \b}{H^{r+5}}^2+C\norm{
		\ta}{H^{r+5}}^2+   \frac18\norm{ \nabla a}{H^{r+3}}^2.\label{tt1}
\end{align}
The time integrability of $\norm{{\n}\cdot\nabla \u}{H^{r+3}}^2$ follows as a
special consequence.

\vskip .1in
The time integrability of $\norm{\div\u}{H^{{r+3}}}^2$ doesn't appear to be
trivial due to the fact that   the equation of
$\u$ is inviscid. However, by exploring the interaction of $\div\u$ and
$\ta$, we are
able to capture the wave structure. Again we explain this discovery in terms of
the linearized system,
\begin{align*}
	&\partial_t \div\u+ R\Delta a+ R\Delta \ta
	=-\Delta ({\n}\cdot\b),\\
	&\partial_t\theta-\kappa\Delta\ta+ \div \u=0.
\end{align*}
We converted into the following wave equations
\begin{align*}
	& \partial_{tt} a - R \Delta a = R \Delta \theta - \Delta({\n}\cdot\b) ,\\
	& \partial_{tt} \div \u - R\Delta\div \u =- R\Delta \partial_t\ta
	-\Delta ({\n}\cdot\partial_t\b),
\end{align*}
which allows us to gain the following time integrability inequality
\begin{align}
	&\norm{\div\u}{H^{{r+3}}}^2+\sum_{0\le s\le
	r+3}\frac{d}{dt}{\big\langle{\Lambda^{s}}
	\ta,{\Lambda^{s}}\div\u\big\rangle}
	\nn\\
	&\quad\le(\frac18+\delta^2)\norm{ \nabla a}{H^{r+3}}^2+C\norm{
	\ta}{H^{r+5}}^2+C\delta^2\norm{ (\ta,\b)}{H^{r+4}}^2+ C\delta^2
	\norm{{\n}\cdot\nabla \u}{H^{r+3}}^2. \label{tt2}
\end{align}
Equation \eqref{tt2} reflects the influence of temperature on the divergence of the velocity field. This relation is also physically meaningful, as temperature directly affects the compressibility of the fluid,  governing how it expands or contracts through the divergence of the velocity field.

\vskip .1in
Having gained the bounds in (\ref{tt}), (\ref{tt1}) and (\ref{tt2}), the
strategy next is to prove an energy estimate of the form, for any integer
$\ell\ge
0$,
\begin{align}
	&\frac12\frac{d}{dt}\left(\norm{(a,\u,\ta,\b)}{H^{\ell}}^2
	+\int_{\T^3}\frac{\ta+a^2}{(1+a)^2}(\la^{\ell}a)^2 \,dx
	\right)
	+\kappa\norm{\nabla\ta}{H^{{\ell}}}^2
	+\sigma\norm{\nabla\b}{H^{{\ell}}}^2\nn\\
	&\qquad\le C
	Y_{\infty}(t)\norm{(a,\u,\ta,\b)}{H^{\ell}}^2, \label{tt3}
\end{align}
where $Y_{\infty}(t)$ essentially contains the $L^\infty$-norm of the low-order
derivatives, namely
\begin{align*}
	Y_{\infty}(t)= &
	\|(a,\u,\ta,\b)\|_{L^{\infty}}+(1+\|a\|_{L^{\infty}}^2)\|(a,\u,\ta,\b)\|_{L^{\infty}}^2
	+(1+
	\norm{a}{L^\infty})\norm{(\nabla a,\nabla\u,\nabla
		\ta,\nabla\b)}{L^\infty}\nn\\
	&+\norm{\Delta
		\ta}{L^\infty}+(1+\|(a,\u,\b)\|_{L^{\infty}}^2+\|\nabla\u\|_{L^{\infty}}^2)\norm{(\nabla
		a,\nabla\u,\nabla \ta,\nabla\b)}{L^\infty}^2.
\end{align*}
A very technical spot in the proof of \eqref{tt3} is due to the lack of
dissipation in the equation of $\u$. More precisely, when we estimate the
Sobolev norm $\|a\|_{H^\ell}$, we need to deal with the term
$$
\int_{\T^3} a \Lambda^{\ell} \div \u \, \Lambda^{\ell} a\,dx,
$$
which generates $(\ell+1)$-derivative on $\u$. This difficult situation is dealt
with by substituting the equation
$$
 \div\u=-\frac{\partial_t a+\u\cdot\nabla a}{1+a},
$$
which helps increase the degree of nonlinearity and spread the derivatives.
More technical details can be found in Section \ref{sec:high}.

\vskip .1in
A suitable combination of (\ref{tt}), (\ref{tt1}) and (\ref{tt2}) with
\eqref{Sob} and \eqref{tt3} allows us to control the right-hand side of
(\ref{tt3}) and convert \eqref{tt3} into an equation of the form
\begin{align*}
	\frac{d}{dt}{\mathcal{E}(t)}+c({\mathcal{E}(t)})^{\frac43}\le 0,
\end{align*}
which yields the desired decay rates. The precise definition of $\mathcal{E}$
is given in Section \ref{sec:pro}.

\vskip .1in
The rest of this paper is divided into seven sections. Section \ref{sec:l2}
presents the global uniform (in time) $L^2$-bound on the solution of
\eqref{m2}. Section \ref{sec: tap} prepares a generalized Poincar\'e type
inequality for $\ta$. Section \ref{sec:high} derives the energy estimate on
the solution of \eqref{m2} in the Sobolev space $H^\ell$. The main result
is stated in Proposition \ref{high2}. Section \ref{sec:adiss} discovers
and exploits the wave structure in the coupled system of $a$ and $\div\u$.
The main result is the estimate in \eqref{tt}. Section \ref{sec:nudiss}
combines the equations of $\u$ and $\b$ to derive the wave structure and
establish \eqref{tt1}. Section \ref{sec:dudiss} makes use of the equations of
$\div\u$ and $\theta$ to obtain the wave structure and thus prove \eqref{tt2}.
The last section combines the estimates above and apply the bootstrapping
argument to finish the proof of our main result.

\vskip .3in
\section{Global and uniform $L^2$-bound}
\label{sec:l2}

This section presents the global $L^2$ bound on $(a,\u,\ta,\b)$. The following
Sobolev space inequalities will be used frequently.

\begin{lemma}\label{daishu}{\rm(\cite{kato})} For any $s\ge 0$, there exists a
	positive constant $C=C(s)$ such that, for any $f,g\in {H^{s}}(\T^3)\cap
	{L^\infty}(\T^3)$, we have
	\begin{equation}\label{mor}
		\|fg\|_{H^{s}}\le
		C(\|f\|_{L^\infty}\|g\|_{H^{s}}+\|g\|_{L^\infty}\|f\|_{H^{s}}).
	\end{equation}
\end{lemma}
\begin{lemma}\label{jiaohuanzi}{\rm(\cite{kato})}
	For any $s\ge 0$, there exists a
	positive constant $C=C(s)$ such that, for any $f\in {H^{s}}(\T^3)\cap
	W^{1,\infty}(\T^3)$, $g\in
	{H^{s-1}}(\T^3)\cap {L^\infty}(\T^3)$, there holds
	\begin{align*}
		\norma{[\la^s,f\cdot\nabla ]g}{L^2}\le C(\norm{\nabla
		f}{L^\infty}\norm{\la^sg}{L^2}+\norm{\la^s f}{L^2}\norm{\nabla
		g}{L^\infty}).
	\end{align*}
\end{lemma}

\begin{lemma}\label{fuhe}{\rm(\cite{Triebel})}
	Let $s>0$, $f\in H^s(\T^3)\cap L^\infty(\T^3)$. Assume that $F$ is a smooth
	function  on $\R$ with $F(0)=0$. Then we
	have
	$$
	\|F(f)\|_{H^s}\le C(1+\|f\|_{L^\infty})^{[s]+1}\|f\|_{H^s}
	$$
	where the constant $C$ depends on $\sup_{k\le{[s]+2},t\le\|f\|_{L^\infty}}
	\|F^k(t)\|_{L^\infty}.$
\end{lemma}

\vskip .1in
When a vector  ${\n}\in\R^3$ satisfies the Diophantine condition
\eqref{diufantu}, the Sobolev norm of the directional derivative of any
function $f$ along ${\n}\in\R^3$ can actually control a lower-order Sobolev
norm of $f$. The
precise statement is given in the following lemma.

\begin{lemma}\label{dir}
Let ${\n}\in\R^3$ satisfy the Diophantine condition \eqref{diufantu}.
\begin{itemize}
  \item For any $s\in \mathbb R$,  if $\int_{\T^3}f\,dx=0$,
  there holds
  \begin{equation}\label{poin1}
\|f\|_{H^{s}}\le C\|{\n}\cdot\nabla f\|_{H^{s+r}}.
\end{equation}
  \item For any $s>0$, one can remove the zero-mean condition by using
  homogeneous norms. That is, if $s>0$, there holds, for any $f$, that
\begin{equation}\label{poin2}
\|f\|_{\dot{H}^{s}}\le C\|{\n}\cdot\nabla f\|_{H^{s+r}}.
\end{equation}
\end{itemize}
\end{lemma}
\begin{proof}
We give the proof for completeness. By Plancherel's formula,
\begin{align*}
\|{\n}\cdot\nabla f\|_{H^{s+r}}^2=&\sum_{{\mathbf{k}}\in\Z^3}(1+|{\mathbf{k}}|^2)^{s+r}|\n\cdot{\mathbf{k}}|^2|\hat{f}|^2\\
=&\sum_{{\mathbf{k}}\in\Z^3\setminus \{\mathbf{0}\}}(1+|{\mathbf{k}}|^2)^{s+r}|\n\cdot{\mathbf{k}}|^2|\hat{f}|^2\\
\ge&c\sum_{{\mathbf{k}}\in\Z^3\setminus \{\mathbf{0}\}}(1+|{\mathbf{k}}|^2)^{s+r}|{\mathbf{k}}|^{-2r}|\hat{f}|^2\\
\ge&c\sum_{{\mathbf{k}}\in\Z^3\setminus \{\mathbf{0}\}}(1+|{\mathbf{k}}|^2)^{s}|\hat{f}|^2.
\end{align*}
So if $\int_{\T^3}f\,dx =0$, we have \eqref{poin1} since $\hat{f}(0)=0$. If $s>0$, we have \eqref{poin2}.
\end{proof}

Due to the lack of zero-mean condition for $\u$, we  need to
 generalize the above lemma by a direct perturbation technique.
\begin{lemma}
Let ${\n}\in\R^3$ satisfy \eqref{diufantu} and $\rho\in L^2(\T^3)$ satisfy
 \begin{equation}\label{pingjun}
\|\rho-1\|_{L^2}\le\frac12, \quad\hbox{and} \int_{\T^3}\rho {\u}\,dx=0.
 \end{equation}
 Then for any $s\ge0$,
\begin{equation}\label{poin3}
\|{\u}\|_{H^{s}}\le C\|{\n}\cdot\nabla {\u}\|_{H^{s+r}}.
\end{equation}
\end{lemma}
\begin{proof}
For any $s>0,$ there holds
$$ \|{\u}\|_{H^{s}}\approx \|{\u}\|_{L^{2}}+\|{\u}\|_{\dot{H}^{s}}.$$
Hence, in view of  \eqref{poin2}, we have
 \begin{align}\label{poin4}
 \|{\u}\|_{H^{s}}\approx& \|{\u}\|_{L^{2}}+\|{\u}\|_{\dot{H}^{s}}\nonumber\\
 \lesssim&\|{\u}\|_{L^{2}}+\|{\n}\cdot\nabla {\u}\|_{H^{s+r}}.
 \end{align}
 Next, we only need to verify \eqref{poin3} holds for $s=0$. Denote $\bar{\u}$ the mean of ${\u}$, indeed by \eqref{poin1}, there holds
\begin{align*}
\|{\u}\|_{L^2}\le\|{\u}-\bar{{\u}}\|_{L^2}+\|\bar{{\u}}\|_{L^2}\le C\|{\n}\cdot\nabla {\u}\|_{H^{r}}+|\bar{{\u}}|.
\end{align*}
But we note from \eqref{pingjun} that
\begin{align*}
|\bar{{\u}}|=\left|\int_{\T^3}(\rho-1){\u}\,dx\right|\le\|\rho-1\|_{L^2}\|{\u}\|_{L^2}\le\frac 12\|{\u}\|_{L^2}.
\end{align*}
Putting two estimates together implies that,
\begin{align}\label{poin5}
\|{\u}\|_{L^2}\le C\|{\n}\cdot\nabla {\u}\|_{H^{r}},
\end{align}
from which and \eqref{poin4}, we arrive at \eqref{poin3}.
This finishes the proof of the lemma.
\end{proof}

\vskip .1in
Throughout the paper, without loss of generality, we assume that $R = 1$.
The goal of this section is to show that any solution of \eqref{m2} satisfies
the following uniform global $L^2$-bound.

\begin{proposition}\label{basic}
Let $(a,\u,\ta,\b) \in C([0, \infty];H^N)$ be a solution to
\eqref{m2}. Then, for any $t\ge 0$,
\begin{align}
\frac12\frac{d}{dt}\norma{(a,\u,\b,\ta)}{L^2}^2+\sigma\int_{\T^3}\frac{|\nabla
	\times\b|^2}{\vta} {\,dx} +\kappa\int_{\T^3}\frac{|\nabla
	\vta|^2}{|\vta|^2}
{\,dx} \le 0. \label{ping60}
\end{align}
As a consequence, if
$
c_0\le\rho, \vta\le c_0^{-1} \mbox{for fixed positive constant $c_0$},
$
then
\begin{align}\label{ping6}
\frac{d}{dt}\norm{(a,\u,\b,\ta)}{L^2}^2+\sigma\norm{\nabla\b}{L^2}^2
+\kappa\norm{\nabla\ta}{L^2}^2\le 0.
\end{align}
\end{proposition}

\begin{proof}
Integrating the mass equation $\eqref{m2}_1$  over $\T^3$ implies
\begin{align}\label{eq3.4}
\int_{\T^3}\rho\div \u {\,dx}&=-\int_{\T^3}\rho((\ln \rho)_t+\u\cdot\nabla\ln \rho)\,dx\nn\\
&=-\frac{d}{dt}\int_{\T^3} \rho\ln \rho \,dx=-\frac{d}{dt}\int_{\T^3}
(\rho\ln \rho-\rho+1){\,dx},
\end{align}
where we have used (\ref{conser1}) in the last equation.
Then, multiplying the momentum equation $\eqref{m2}_2$ by $\u$ and integrating
by parts, we have
\begin{align}\label{eq3.5}
&\frac12\frac{d}{dt}\int_{\T^3}\rho{|\u|^2}{\,dx}-\int_{\T^3}\rho\vta\div \u {\,dx}\nn\\
&\quad=\int_{\T^3} {\b}\cdot\nabla {\b}\cdot \u{\,dx}-\int_{\T^3}{\b}\nabla {\b}\cdot \u{\,dx}+\int_{\T^3}{\n}\cdot\nabla \b\cdot \u{\,dx}-\int_{\T^3}{\n}\nabla \b\cdot \u{\,dx}.
\end{align}
Next, integrating the energy equation $\eqref{m2}_3$, integrating the product
of the mass equation $\eqref{m2}_1$ with $\vta$, and summing up the
resultants, we get
\begin{align}\label{eq3.6}
&\frac{d}{dt}\int_{\T^3} \rho\vta {\,dx}+\int_{\T^3}\rho\vta\div \u
{\,dx}=\sigma \int_{\T^3}|\nabla \times\b|^2 {\,dx}.
\end{align}
Multiplying the energy equation $\eqref{m2}_3$ by $\vta^{-1}$ and then
integrating by parts, multiplying the mass equation $\eqref{m2}_1$
by $\ln \vta$ and summing up the resultants, we obtain
\begin{align}\label{eq3.7}
&-\frac{d}{dt}\int_{\T^3} \rho\ln\vta {\,dx}+\kappa\int_{\T^3}\frac{|\nabla \vta|^2}{|\vta|^2} {\,dx}-\int_{\T^3}\rho\div \u {\,dx}\nn\\
&\quad=-\sigma\int_{\T^3}\frac{1}{\vta}\big(|\nabla \times\b|^2\big) {\,dx}.
\end{align}
Multiplying the magnetic equation $\eqref{m2}_4$ by ${\b}$, and integrating by
parts, we find
\begin{align}\label{eq3.8}
&\frac12\frac{d}{dt}\int_{\T^3} |{\b}|^2 {\,dx}+\sigma\int_{\T^3}
|\nabla{\b}|^2 {\,dx}+\int_{\T^3} \u\cdot\nabla {\b}\cdot
{\b}{\,dx}+\int_{\T^3} {\b} \div \u\cdot {\b}{\,dx}\nn\\
&\quad=\int_{\T^3} {\b}\cdot\nabla \u\cdot {\b}{\,dx}+\int_{\T^3} {\n}\cdot\nabla \u\cdot {\b}{\,dx}-\int_{\T^3} {\n} \div \u\cdot {\b}{\,dx}.
\end{align}
Since ${\b}$ is divergence free, it is easy to check that
\begin{align*}
&\int_{\T^3} ({\n}\cdot\nabla \b+{\b}\cdot\nabla {\b})\cdot \u\,{\,dx}+\int_{\T^3} ({\n}\cdot\nabla \u+{\b}\cdot\nabla \u)\cdot {\b}\,{\,dx}=0,\nonumber\\
&\int_{\T^3} ({\n}\nabla \b+{\b}\nabla \b)\cdot \u\,{\,dx}+\int_{\T^3} \u\cdot\nabla {\b}\cdot {\b}\,{\,dx}+\int_{\T^3}  ({\b} \div \u+{\n} \div \u)\cdot {\b}\,{\,dx}=0.
\end{align*}
Thus, putting \eqref{eq3.4}--\eqref{eq3.8} together gives \eqref{ping60}. If,
for fixed positive constant $c_0$,
$$
c_0\le\rho, \vta\le c_0^{-1},
$$
then
\begin{align}\label{eq3.9}
&\frac{d}{dt}\left(\frac{1}{2}\int_{\T^3}\rho|\u|^2{\,dx}+\int_{\T^3} (\rho\ln \rho-\rho+1){\,dx}+\int_{\T^3} \rho(\vta-\ln\vta-1) {\,dx}+\frac12\int_{\T^3} |{\b}|^2 {\,dx}\right)\nn\\
&\quad+C\left(\|\nabla \b\|_{L^2}^2+\|\nabla \vta\|_{L^2}^2\right)\le 0.
\end{align}
By the Taylor expansion,
\begin{align*}
\rho\ln \rho-\rho+1\sim(\rho-1)^2,\>\hbox{ and }\>
\rho(\vta-\ln\vta-1)\sim(\vta-1)^2 \quad \hbox{as}\quad
\rho\rightarrow1\>\hbox{ and }\>\vta\rightarrow1.
\end{align*}
Then,  \eqref{ping6} follows as a consequence of \eqref{eq3.9}.
\end{proof}

\vskip .3in
\section{A generalized Poincar\'e inequality for $\ta$}
\label{sec: tap}

This section is devoted to proving the following Poincar\'e type inequality for
$\ta$. Without loss of generality, we set $$|\mathbb T^3|=1.$$ We assume the
initial data $(\rho_0,\u_0,\b_0,\ta_0)$ satisfies \eqref{conser1}, namely
\begin{align}
	&\int_{\mathbb T^3} \rho_0(x) \,dx =1,\quad\int_{\mathbb T^3} \rho_0(x)
	\u_0(x) \,dx =\int_{\mathbb T^3} \b_0(x) \,dx =0.\label{conser0}
\end{align}
In addition, we assume that
$$
\int_{\mathbb T^3} \rho_0 \ta_0 \,dx+\frac12\int_{\mathbb T^3} \rho_0
|\u_0|^2dx+\frac12\int_{\mathbb T^3} |\b_0|^2\,dx=0.
$$
Owing to the conservation of total mass,
total momentum, and total energy,  we
have, for any $t\ge0,$ that
\begin{align}
	&\int_{\mathbb T^3} \rho(x) \,dx =1,\quad\int_{\mathbb T^3} \rho(x) \u(x)
	\,dx =\int_{\mathbb T^3} \b(x) \,dx =0,\label{conser2}\\
	&\int_{\mathbb T^3} \rho \ta \,dx+\frac12\int_{\mathbb T^3} \rho
	|\u|^2dx+\frac12\int_{\mathbb T^3} |\b|^2\,dx=0.\label{conser3}
\end{align}
The generalized version of the Poincar\'e type inequality for $\ta$ can be
stated as follows.

\begin{lemma}\label{poincareineq}
	Let $(\rho,\u, \ta, \b)$ be smooth solutions to \eqref{m2} and satisfy
	\eqref{conser2}, \eqref{conser3} and
	\begin{align}\label{jinq0}
		c_0\le\rho, {\ta}\le c_0^{-1}.
	\end{align}
There exists a positive constant $C$, depending on $\Omega$ and $c_0$, such that
	\begin{align}\label{jinq1}
		\|\ta\|_{L^2}^2\leq& \,
		C\|\nabla\ta\|_{L^2}^2+C\|\nabla\u\|_{L^2}^4+C\|\nabla\b\|_{L^2}^4.
	\end{align}
\end{lemma}

\vskip .1in
To prepare for the proof, we present several functional inequalities.
We first recall a weighted Poincar\'e inequality first established by
Desvillettes and Villani in \cite{DV05}.
\begin{lemma}\label{lem-Poi}
	Let $\Omega$  be a bounded connected Lipschitz domain and $\bar{\varrho}$
	be a positive constant.  There exists a positive constant $C$, depending on
	$\Omega$ and $\bar{\varrho}$, such that  for any nonnegative function
	$\varrho$ satisfying
	\begin{align*}
		\int_{\Omega}\varrho dx=1, \quad \varrho\le\bar{\varrho},
	\end{align*}
	and any $f\in H^1(\Omega)$, there holds
	\begin{align}\label{wPoi}
		\int_{\Omega}\varrho \left(f-\int_{\Omega}\rho f \,dx\right)^2\,dx\le
		C\|\nabla f\|_{L^2}^2.
	\end{align}
\end{lemma}

\vskip .1in
In order to remove the weight function $\varrho$ in  \eqref{wPoi} without
resorting to the lower bound of  $\varrho$, we need  another variant of
Poincar\'e inequality (see Lemma 3.2 in \cite{F04}).
\begin{lemma}\label{lem2.2}
	Let $\Omega$ be a bounded connected Lipschitz domain in $\mathbb{R}^3$ and
	$p>1$ be a constant. Given positive constants $M_0$ and $E_0$, there is a
	constant $C=C(E_0,M_0)$ such that for any
	non-negative function $\varrho$ satisfying
	$$
	M_0\leq\int_{\Omega}\varrho dx\quad\mbox{and}\quad
	\int_{\Omega}\varrho^{p}dx\leq
	E_0,
	$$
	and for any $\u\in H^1(\Omega)$, there holds
	$$
	\|\u\|_{L^2}^2\leq C\left[\|\nabla
	\u\|_{L^2}^2+\left(\int_{\Omega}\varrho|\u|\,dx\right)^2\right].
	$$
\end{lemma}

We are ready to prove Lemma \ref{poincareineq}.
\begin{proof}[Proof of Lemma \ref{poincareineq}]
	First, it's easy to deduce from \eqref{jinq0} that
	\begin{align}\label{jinq3}
		\|\ta\|_{L^2}^2
		\le&\,C\|\sqrt\rho\ta\|_{L^2}^2\nn\\
		=&\,C\int_{\mathbb T^3} \rho \Big|\ta-\int_{\mathbb T^3} \rho \ta
		\,dx+\int_{\mathbb T^3} \rho \ta \,dx\Big|^2 \,dx\nn\\
		\le&\, C\int_{\mathbb T^3} \rho \Big|\ta-\int_{\mathbb T^3} \rho \ta
		\,dx\Big|^2 \,dx
		+C\int_{\mathbb T^3} \rho \Big|\int_{\mathbb T^3} \rho \ta \,dx\Big|^2
		\,dx.
	\end{align}
	By Lemma \ref{lem-Poi},
	\begin{align}\label{jinq4}
		\int_{\mathbb T^3} \rho \Big|\ta-\int_{\mathbb T^3} \rho \ta
		\,dx\Big|^2 \,dx\le C\|\nabla\ta\|_{L^2}^2.
	\end{align}
	Thanks to \eqref{conser3}, we have
	\begin{align*}
		\int_{\mathbb T^3} \rho \ta \,dx=-\frac12\int_{\mathbb T^3} \rho
		|\u|^2dx-\frac12\int_{\mathbb T^3} |\b|^2\,dx,
	\end{align*}
Then the last  term in \eqref{jinq3} can be bounded as follows,
	\begin{align}\label{jinq5}
		\int_{\mathbb T^3} \rho \Big|\int_{\mathbb T^3} \rho \ta \,dx\Big|^2
		\,dx
		=& \Big|\int_{\mathbb T^3} \rho \ta
		\,dx\Big|^2=\Big|\frac12\|\sqrt\rho\u\|_{L^2}^2+\frac12\|\b\|_{L^2}^2\Big|^2\nn\\
		\le&C\|\u\|_{L^2}^4+\|\b\|_{L^2}^4.
	\end{align}
	Due to
	$\int_{\mathbb T^3} \rho(x) \u(x) \,dx=0$, one can deduce from Lemma
	\ref{lem-Poi} that
	\begin{align*}
		\|\sqrt{\rho} \u\|_{L^2}^2\le C\|\nabla \u\|_{L^2}^2,
	\end{align*}
	which combines with Lemma \ref{lem2.2} imply that
	\begin{align}\label{jinq7}
		\|\u\|_{L^2}^2\le C\|\nabla\u\|_{L^2}^2.
	\end{align}
	Inserting \eqref{jinq7} into \eqref{jinq5} and using $\int_{\mathbb T^3}
	\b(x) \,dx =0$, we get
	\begin{align}\label{jinq8}
		\int_{\mathbb T^3} \rho \Big|\int_{\mathbb T^3} \rho \ta \,dx\Big|^2
		\,dx
		\le C\|\nabla\u\|_{L^2}^4+ C\|\nabla\b\|_{L^2}^4.
	\end{align}
Inserting \eqref{jinq4} and \eqref{jinq8} in \eqref{jinq3}, we obtain
	\eqref{jinq1}.
This completes the proof of Lemma \ref{poincareineq}.
\end{proof}

\vskip .3in
\section{High-order energy estimates}
\label{sec:high}

This section derives the high-order
energy estimates. Throughout this section, we assume that $a$ and $\ta$ satisfy
\beq\label{asu}
\sup_{x\in \mathbb T^3, t>0} |a(t,x)| \le \frac12, \quad \sup_{x\in \mathbb
T^3, t>0} |\ta(t,x)| \le \frac12.
\eeq
(\ref{asu}) is ensured  by the fact that the solutions constructed here has
small norm in $H^2(\mathbb T^3)$. (\ref{asu}) allows us to freely use several
inequalities such as the composition estimate stated in Lemma \ref{fuhe}, for
any smooth function $G$ with $G(0)=0$,
\beq \label{eq:smalla}
\|G(a)\|_{H^s} \le C\, \|a\|_{H^s} \quad\mbox{for any $s>0$}.
\eeq

\vskip .1in
Our main result is stated in the following proposition.
\begin{proposition}\label{high2}
Let $(a,\u,\ta,\b) \in C([0, T];H^N)$ be a solution to \eqref{m2}. For any
$0\le \ell\le N$, there holds
\begin{align}\label{ed3}
&\frac12\frac{d}{dt}\left(\norma{(a,\u,\ta,\b)}{H^{\ell}}^2
+\int_{\T^3}\frac{\ta+a^2}{(1+a)^2}(\la^{\ell}a)^2 \,dx
\right)
	+\sigma\norm{\nabla\b}{H^{{\ell}}}^2+\kappa\norm{\nabla\ta}{H^{{\ell}}}^2\nn\\
&\qquad\le C
Y_{\infty}(t)\norm{(a,\u,\ta,\b)}{H^{\ell}}^2
\end{align}
with
\begin{align}\label{ed3+1}
Y_{\infty}(t)\stackrel{\mathrm{def}}{=}&
\|(a,\u,\ta,\b)\|_{L^{\infty}}+(1+\|a\|_{L^{\infty}}^2)\|(a,\u,\ta,\b)\|_{L^{\infty}}^2
+(1+
\norm{a}{L^\infty})\norm{(\nabla a,\nabla\u,\nabla \ta,\nabla\b)}{L^\infty}\nn\\
&+\norm{\Delta \ta}{L^\infty}+(1+\|(a,\u,\b)\|_{L^{\infty}}^2+\|\nabla\u\|_{L^{\infty}}^2)\norm{(\nabla a,\nabla\u,\nabla \ta,\nabla\b)}{L^\infty}^2.
\end{align}
\end{proposition}

\begin{proof}
To prove (\ref{ed3}), we reformulate \eqref{m2} by separating the linear terms
from the nonlinear ones. Setting
$$
\bar{\kappa}(\rho)\stackrel{\mathrm{def}}{=}\frac{\kappa}{\rho},
\quad I(a)\stackrel{\mathrm{def}}{=}\frac{a}{1+a},\quad\hbox{and}\quad J(a)=\ln({1+a} ),$$
we have
\begin{eqnarray}\label{tam3}
\left\{\begin{aligned}
&\partial_ta+ \div\u  ={F_1},\\
&\partial_t \u+\nabla a+\nabla \ta
={\n}\cdot\nabla \b-\nabla ({\n}\cdot\b)+{F_2},\\
&\partial_t\theta-\div(\bar{\kappa}(\rho)\nabla\ta)+ \div \u={F_3},\\
&\partial_t \b-\sigma \Delta \b={\n}\cdot\nabla \u-\n\div\u+{F_4},\\
&\div \b =0,\\
&(a,\u,\ta,\b)|_{t=0}=(a_0,\u_0,\ta_0,\b_0),
\end{aligned}\right.
\end{eqnarray}
where
\begin{align*}
{F_1}\stackrel{\mathrm{def}}{=}&-\u\cdot\nabla a-a\div \u,\nn\\
{F_2}\stackrel{\mathrm{def}}{=}&-\u\cdot\nabla \u+\b\cdot\nabla\b+\b\nabla\b+I(a)\nabla a-\theta\nabla J(a)
\nn\\
&-I(a)({\n}\cdot\nabla \b+\b\cdot\nabla\b-{\n}\nabla \b-\b\nabla\b),\nn\\
{F_3}\stackrel{\mathrm{def}}{=}&- \div(\theta \u)-\kappa(\nabla I(a))\nabla\ta+
\frac{ \sigma|\nabla\times   \   \b|^2}{1+a},\nn\\
{F_4}\stackrel{\mathrm{def}}{=}&-\u\cdot\nabla\b+\b\cdot\nabla\u-\b\div\u.
\end{align*}

For $\ell=0$,  \eqref{ping6} implies that
\begin{align*}
\frac12\frac{d}{dt}\norm{( a,\u,\ta,\b)}{L^2}^2+\sigma\norm{\nabla\b}{L^2}^2
+\kappa\norm{\nabla\ta}{L^2}^2\le0.
\end{align*}
We now set $\ell\ge 1$.
For any $1\le s\le \ell$, applying  $\la^s$ to both sides of \eqref{tam3} and
then taking the $L^2$ inner product with $\la^sa, \la^s\u, \la^s\ta,\la^s\b$
respectively gives
\begin{align*}
&\frac12\frac{d}{dt}\norm{(\la^sa,\la^s\u,\la^s\ta,\la^s\b)}{L^2}^2
-\int_{\T^3}\la^{s}\div(\bar{\kappa}(\rho)\nabla\ta)\cdot\la^{s} \ta\,dx+\sigma\int_{\T^3}|\la^{s} \nabla\b|^2\,dx \nn\\
&\quad= \int_{\T^3}\la^{s} {F_1}\cdot\la^{s} a\,dx+\int_{\T^3}\la^{s} {F_2}\cdot\la^{s} \u\,dx+\int_{\T^3}\la^{s} {F_3}\cdot\la^{s} \ta\,dx+\int_{\T^3}\la^{s} {F_4}\cdot\la^{s} \b\,dx,
\end{align*}
where we used the  following cancellations
\begin{align*}
&\int_{\T^3}\la^s \div\u\cdot\la^s a\,dx+\int_{\T^3}\la^s \nabla a\cdot\la^s \u\,dx=0;\nn\\
&\int_{\T^3}\la^s \nabla \ta\cdot\la^s \u\,dx+\int_{\T^3}\la^s \div\u\cdot\la^s \ta\,dx=0;\nn\\
&\int_{\T^3}\la^s({\n}\cdot\nabla \b)\cdot\la^s \u\,dx+\int_{\T^3}\la^s ({\n}\cdot\nabla \u)\cdot\la^s \b\,dx=0;\nn\\
&\int_{\T^3}\la^s\nabla ({\n}\cdot\b)\cdot\la^s \u\,dx+\int_{\T^3}\la^s (\n\div\u)\cdot\la^s \b\,dx=0.
\end{align*}

The second term of the left-hand side can be written as
\begin{align}\label{bianx1}
&-\int_{\T^3}\la^{s}\div(\bar{\kappa}(\rho)\nabla\ta)\cdot\la^{s} \ta\,dx\nn\\
&\quad
=\int_{\T^3}\la^{s}(\bar{\kappa}(\rho)\nabla\ta)\cdot\nabla\la^{s} \ta\,dx\nn\\
&\quad= \int_{\T^3}\bar{\kappa}(\rho)\nabla\la^{s}\ta\cdot\nabla\la^{s} \ta\,dx+\int_{\T^3}[\la^{s},\bar{\kappa}(\rho)]\nabla\ta\cdot\nabla\la^{s} \ta\,dx.
\end{align}
Due to \eqref{asu}, we have for any $t\in[0,T]$ that
\begin{align}\label{bianx2}
 \int_{\T^3}\bar{\kappa}(\rho)\nabla\la^{s}\ta\cdot\nabla\la^{s} \ta\,dx\ge c_0^{-1}\kappa\norm{\la^{s+1}\ta}{L^{2}}^2.
\end{align}
For the last term in \eqref{bianx1}, we first rewrite it into
\begin{align*}
\int_{\T^3}[\la^{s},\bar{\kappa}(\rho)]\nabla\ta\cdot\nabla\la^{s} \ta\,dx
=&\int_{\T^3}[\la^{s},\bar{\kappa}(\rho)-\kappa+\kappa]\nabla\ta\cdot\nabla\la^{s} \ta\,dx\nn\\
=&-\int_{\T^3}[\la^{s},\kappa I(a)]\nabla\ta\cdot\nabla\la^{s} \ta\,dx.
\end{align*}
Then, with the aid of  \eqref{eq:smalla}, we have
\begin{align}
&\Big|\int_{\T^3}[\la^{s},\kappa I(a)]\nabla\ta\cdot\nabla\la^{s}
\ta\,dx\Big|\notag\\
&\quad\le \,C\norm{\nabla\la^{s}\ta}{L^{2}}(\norm{\nabla I(a)
}{L^\infty}\norm{\la^{s} \ta}{L^2}+\norm{\nabla
\ta}{L^\infty}\norm{\la^{s}I(a)}{L^2})\nn\\
&\quad\le \,\frac{c_0^{-1}}{2}\kappa\norm{\la^{s+1}\ta}{L^{2}}^2+C(\norm{\nabla a
}{L^\infty}^2\norm{\la^{s} \ta}{L^2}^2+\norm{\nabla
\ta}{L^\infty}^2\norm{\la^{s}a}{L^2}^2),\label{bianx5}
\end{align}
where we have used Lemma \ref{jiaohuanzi}.  Inserting \eqref{bianx2} and
\eqref{bianx5} in \eqref{bianx1} leads to
\begin{align*}
-\int_{\T^3}\la^{s}\div(\bar{\kappa}(\rho)\nabla\ta)\cdot\la^{s} \ta\,dx
\ge&\frac{c_0^{-1}}{2}\kappa\norm{\la^{s+1}\ta}{L^{2}}^2\nn\\
&-C(\norm{\nabla a }{L^\infty}^2\norm{\la^{s} \ta}{L^2}^2+\norm{\nabla \ta}{L^\infty}^2\norm{\la^{s}a}{L^2}^2).
\end{align*}
Therefore,
\begin{align}\label{ed4}
&\frac{d}{dt}\norm{(\la^sa,\la^s\u,\la^s\ta,\la^s\b)}{L^2}^2
+
2\sigma\norm{\la^{s+1}\b}{L^2}^2+{c_0^{-1}}\kappa\norm{\la^{s+1}\ta}{L^2}^2\nn\\
&\quad\le\,
C(\norm{\nabla a }{L^\infty}^2\norm{\la^{s} \ta}{L^2}^2+\norm{\nabla
\ta}{L^\infty}^2\norm{\la^{s}a}{L^2}^2)+
C\int_{\T^3}\la^{s} {F_1}\cdot\la^{s} a\,dx\nn\\
&\qquad+C\int_{\T^3}\la^{s} {F_2}\cdot\la^{s} \u\,dx+C\int_{\T^3}\la^{s} {F_3}\cdot\la^{s} \ta\,dx+C\int_{\T^3}\la^{s} {F_4}\cdot\la^{s} \b\,dx.
\end{align}
In the following, we estimate successively each of terms on the right hand side
of \eqref{ed4}. For the first term in ${F_1}$, we rewrite it into
\begin{align}\label{ed5-1}
\int_{\T^3}\la^{s} (\u\cdot\nabla a)\cdot\la^{s} a\,dx
=&\int_{\T^3}(\la^{s} (\u\cdot\nabla a)-\u\cdot\nabla\la^{s}a)\cdot\la^{s} a\,dx
+\int_{\T^3}\u\cdot\nabla\la^{s}a\cdot\la^{s} a\,dx\nn\\
\stackrel{\mathrm{def}}{=}&A_1+A_2.
\end{align}
By Lemma \ref{jiaohuanzi}, one has
\begin{align}\label{we1}
A_1\le& C\norm{[\la^s,\u\cdot\nabla ]a}{L^2}\norm{\la^{s}a}{L^2}\nn\\
\le&C(\norm{\nabla \u}{L^\infty}\norm{\la^{s}a}{L^2}+\norm{\la^{s} \u}{L^2}\norm{\nabla a}{L^\infty})\norm{\la^{s}a}{L^2}\nn\\
\le&C(\norm{\nabla \u}{L^\infty}+\norm{\nabla a}{L^\infty})\|(a,\u)\|_{H^{\ell}}^2.
\end{align}
By integration by parts,
\begin{align}\label{we2}
A_2\le& \,C\norm{\nabla \u}{L^\infty}\|a\|_{H^{\ell}}^2.
\end{align}
To control
the second term in ${F_1}$, we first write
\begin{align*}
\int_{\T^3}\la^{s} (a\div\u)\cdot\la^{s} a\,dx
=&\sum_{1\le s\le \ell-1}\int_{\T^3}\la^{s} (a\div\u)\cdot\la^{s}
a\,dx+\sum_{s= \ell}\int_{\T^3}\la^{s} (a\div\u)\cdot\la^{s} a\,dx.
\end{align*}
The first term on the right-hand side can be bounded by
\begin{align*}
\sum_{1\le s\le \ell-1}\int_{\T^3}\la^{s} (a\div\u)\cdot\la^{s} a\,dx
\le & \,C \sum_{1\le s\le
\ell-1}(\|a\|_{L^\infty}\|\div\u\|_{H^{s}}+\|\div\u\|_{L^\infty}\|a\|_{H^{s}})\|a\|_{H^{s}}\nn\\
\le &C(\|a\|_{L^\infty}\|\u\|_{H^{\ell}}+\|\div\u\|_{L^\infty}\|a\|_{H^{\ell}})\|a\|_{H^{\ell}}\nn\\
\le &C(\|a\|_{L^\infty}\|(a,\u)\|_{H^{\ell}}^2+\|\nabla\u\|_{L^\infty}\|a\|_{H^{\ell}}^2).
\end{align*}
The estimate of the second term isn't straightforward due to $\ell +1$
derivatives on $\u$. The goal here is to reduce the number of derivatives to
$\ell$
\begin{align}\label{yihan2}
&\sum_{s= \ell}\int_{\T^3}\la^{s} (a\div\u)\cdot\la^{s} a\,dx\nn\\
=&\sum_{0\le \alpha\le {\ell-1}}\int_{\T^3}\la^{\ell-\alpha} a\la^{\alpha}\div\u\cdot\la^{\ell} a\,dx+\int_{\T^3}a\la^{\ell} \div\u\cdot\la^{\ell} a\,dx.
\end{align}
It then follows from interpolation inequalities that
 \begin{align}\label{yihan3}
\sum_{0\le \alpha\le {\ell-1}}\int_{\T^3}\la^{\ell-\alpha} a\la^{\alpha}\div\u\cdot\la^{\ell} a\,dx
\le&\norm{\div\u}{L^\infty}\norm{\la^{\ell} a}{L^2}^2
+\norm{\nabla a}{L^\infty}\norm{\la^{\ell-1} \div\u}{L^2}\norm{\la^{\ell} a}{L^2}\nn\\
\le&C\norm{\nabla\u}{L^\infty}\norm{a}{H^\ell}^2
+\norm{\nabla a}{L^\infty}(\norm{\u}{H^\ell}^2+\norm{a}{H^\ell}^2).
\end{align}
To bound the second term on the right-hand side of \eqref{yihan2}, we make use
of the equation
 $$\div\u=-\frac{\partial_t a+\u\cdot\nabla a}{1+a}$$
to obtain
\begin{align*}
\int_{\T^3}a\la^{\ell} \div\u\cdot\la^{\ell} a\,dx
=&-\int_{\T^3}a\la^{\ell} \left(\frac{\partial_t a+\u\cdot\nabla a}{1+a}\right)\cdot\la^{\ell} a\,dx \nn\\
=&-\int_{\T^3}a\la^{\ell} \left(\frac{\partial_t a}{1+a}\right)\cdot\la^{\ell} a\,dx -\int_{\T^3}a\la^{\ell} \left(\frac{\u\cdot\nabla a}{1+a}\right)\cdot\la^{\ell} a\,dx\nn\\
=&D_1+D_2.
\end{align*}
By Leibniz's rule,
\begin{align*}
D_1=&-\int_{\T^3}a\la^{\ell} \left(\frac{\partial_t
a}{1+a}\right)\cdot\la^{\ell} a\,dx = D_{11} + D_{12},
\end{align*}
where
\begin{align*}
D_{11} =&-\int_{\T^3}\frac{a}{1+a}\la^{\ell} \left({\partial_t
a}\right)\cdot\la^{\ell} a\,dx
-\int_{\T^3}a\partial_t a\la^{\ell}\left(\frac{1}{1+a}\right)\cdot\la^{\ell}
a\,dx,\\
D_{12} =&-\sum_{0<\alpha<\ell}\int_{\T^3}a \, \la^{\alpha} \left({\partial_t
	a}\right)\la^{\ell-\alpha}\left(\frac1{1+a}\right) \cdot\la^{\ell} a\,dx.
\end{align*}
By integration by parts,
\begin{align}
D_{11}
=&-\frac12\int_{\T^3}\frac{a}{1+a}\partial_t(\la^{\ell}a)^2 \,dx
-\int_{\T^3}a\partial_t a\la^{\ell}\left(\frac{1}{1+a}\right)\cdot\la^{\ell}
a\,dx\nn\\
=&-\frac12\frac{d}{dt}\int_{\T^3}\frac{a}{1+a}(\la^{\ell}a)^2
\,dx+\frac12\int_{\T^3}\frac{\partial_t a(\la^{\ell}a)^2}{(1+a)^2}\,dx\nn\\
&-\int_{\T^3}a\partial_t a\la^{\ell}\left(\frac{1}{1+a}\right)\cdot\la^{\ell}
a\,dx. \label{d11}
\end{align}
By the equation of $a$,
\begin{align}\label{yihan6}
\frac12\int_{\T^3}\frac{\partial_t a(\la^{\ell}a)^2}{(1+a)^2}\,dx
=&-\frac12\int_{\T^3}\frac{1}{(1+a)^2}(\u\cdot\nabla a+a\div \u+\div
\u)(\la^{\ell}a)^2\,dx\nn\\
\le&C((1+
\norm{a}{L^\infty})\norm{\nabla \u}{L^\infty}+\norm{\nabla a}{L^\infty}\norm{
\u}{L^\infty})\norm{\la^{\ell}a}{L^2}^2.
\end{align}
The last term in \eqref{d11} admits the same bound as the one in \eqref{yihan6}.
Therefore,
\begin{align}\label{yihan7}
D_{11}
\le&-\frac12\frac{d}{dt}\int_{\T^3}\frac{a}{1+a}(\la^{\ell}a)^2 \,dx+C((1+
\norm{a}{L^\infty})\norm{\nabla \u}{L^\infty}+\norm{\nabla a}{L^\infty}\norm{ \u}{L^\infty})\norm{\la^{\ell}a}{L^2}^2.
\end{align}
$D_{12}$ contains terms with intermediate number of derivatives. It is not
difficult to control the terms in $D_{12}$ through interpolation and obtain the
same bound as the one in \eqref{yihan7}. We now bound $D_2$,
\begin{align*}
D_2=&-\int_{\T^3}a\la^{\ell} \left(\frac{\u\cdot\nabla
a}{1+a}\right)\cdot\la^{\ell} a\,dx\nn\\
=&-\int_{\T^3}\frac{a}{1+a}\la^{\ell}(\u\cdot\nabla a)\cdot\la^{\ell} a\,dx-
\int_{\T^3}\frac{a}{1+a}(\u\cdot\nabla
a)\la^{\ell}\left(\frac{1}{1+a}\right)\cdot\la^{\ell}
a\,dx\\
&-\sum_{0<\alpha<\ell}\int_{\T^3}a \, \la^{\alpha} \left({\u\cdot\nabla
a}\right)\la^{\ell-\alpha}\left(\frac1{1+a}\right) \cdot\la^{\ell} a\,dx
&\nn\\
=&D_{2,1}+D_{2,2} +D_{2,3}.
\end{align*}
We rewrite $D_{2,1}$ as
\begin{align*}
D_{2,1}=&-\int_{\T^3}\frac{a}{1+a}\left(\la^{\ell}(\u\cdot\nabla a)-\u\cdot\nabla \la^{\ell} a\right)\cdot\la^{\ell} a\,dx+
\int_{\T^3}\frac{a}{1+a}\u\cdot\nabla\la^{\ell} a\cdot\la^{\ell} a\,dx\nn\\
=&D_{2,1}^{(1)}+D_{2,1}^{(2)}.
\end{align*}
By Lemma \ref{jiaohuanzi},
\begin{align*}
D_{2,1}^{(1)}\le& C\norma{\frac{a}{1+a}}{L^\infty}\norma{[\la^s,\u\cdot\nabla ]a}{L^2}\norm{\la^{\ell}a}{L^2}\nn\\
\le&C(\norm{\nabla \u}{L^\infty}\norm{\la^{\ell}a}{L^2}+\norm{\la^{\ell} \u}{L^2}\norm{\nabla a}{L^\infty})\norm{\la^{\ell}a}{L^2}\nn\\
\le&C(\norm{\nabla \u}{L^\infty}+\norm{\nabla a}{L^\infty})(\norm{\la^{\ell}a}{L^2}^2+\norm{\la^{\ell}\u}{L^2}^2).
\end{align*}
By integration by parts and Lemma \ref{daishu},
\begin{align*}
D_{2,1}^{(2)}\le& C\norma{\div\left(\frac{a\u}{1+a}\right)}{L^\infty}\norm{\la^{\ell}a}{L^2}^2\nn\\
\le&C(\norm{\nabla \u}{L^\infty}+\norm{ \u}{L^\infty}\norm{\nabla a}{L^\infty})\norm{\la^{\ell}a}{L^2}^2.
\end{align*}
By Lemma \ref{fuhe},
\begin{align*}
D_{2,2}
\le&\norm{a}{L^\infty}\norm{ \u}{L^\infty}\norm{\nabla a}{L^\infty}\norm{\la^{\ell}a}{L^2}^2.
\end{align*}
$D_{2,3}$ contains terms with intermediate derivatives and can be estimated by
the bounds of $D_{2,1}$ and $D_{2,2}$.
Therefore,
\begin{align*}
D_{2}
\le&C(\norm{\nabla \u}{L^\infty}+\norm{\nabla
a}{L^\infty}+(1+\norm{a}{L^\infty})\norm{ \u}{L^\infty}\norm{\nabla
a}{L^\infty})(\norm{a}{H^\ell}^2+ \norm{\u}{H^\ell}^2),
\end{align*}
which, together with \eqref{yihan7},  leads to
\begin{align}\label{yihan15}
\int_{\T^3}a\la^{\ell} \div\u\cdot\la^{\ell} a\,dx
\le&-\frac12\frac{d}{dt}\int_{\T^3}\frac{a}{1+a}(\la^{\ell}a)^2 \,dx\nn\\
&+C(\norm{\nabla \u}{L^\infty}+\norm{\nabla a}{L^\infty}+\norm{
\u}{L^\infty}\norm{\nabla a}{L^\infty})(\norm{a}{H^\ell}^2
+ \norm{\u}{H^\ell}^2).
\end{align}
Combining \eqref{yihan3} and \eqref{yihan15} leads to
\begin{align}\label{yihan16}
\int_{\T^3}\la^{s} (a\div\u)\cdot\la^{s} a\,dx
\le&-\frac12\frac{d}{dt}\int_{\T^3}\frac{a}{1+a}(\la^{\ell}a)^2 \,dx\nn\\
&+C(\norm{\nabla \u}{L^\infty}+\norm{\nabla a}{L^\infty}+\norm{ \u}{L^\infty}\norm{\nabla a}{L^\infty})(\norm{a}{H^\ell}^2+\norm{\u}{H^\ell}^2).
\end{align}
\eqref{we1}, \eqref{we2} and \eqref{yihan16} yield
\begin{align*}
\int_{\T^3}\la^{s} {F_1}\cdot\la^{s} a\,dx
\le&-\frac12\frac{d}{dt}\int_{\T^3}\frac{a}{1+a}(\la^{s}a)^2 \,dx\nn\\
&+C(\norm{\nabla \u}{L^\infty}+\norm{\nabla a}{L^\infty}+\norm{ \u}{L^\infty}\norm{\nabla a}{L^\infty})(\norm{\la^{s}a}{L^2}^2+\norm{\la^{s}\u}{L^2}^2).
\end{align*}
We turn to the last term in \eqref{ed4}. For the first term in ${F_4}$, we
obtain via similar estimates as in \eqref{we1} and \eqref{we2},
\begin{align*}
\int_{\T^3}\la^{s} (\u\cdot\nabla \b)\cdot\la^{s} \b\,dx
\le&C(\norm{\nabla \u}{L^\infty}+\norm{\nabla \b}{L^\infty})(\norm{\la^{s}\u}{L^2}^2+\norm{\la^{s}\b}{L^2}^2).
\end{align*}
For the last two terms in ${F_4}$, by Lemma \ref{daishu} and the H\"{o}lder
inequality,
\begin{align*}
\int_{\T^3}\la^{s} (\b\cdot\nabla\u-\b\div\u)\cdot\la^{s} \b\,dx
         \leq&\frac{\sigma}{16}\|\Lambda^{s+1}\b\|_{L^2}^2 +
         C(\|\nabla\u\|_{L^{\infty}}^2+\|\b\|_{L^{\infty}}^2)
         \norm{\la^{s}\b}{L^2}^2.
\end{align*}
As a consequence,
\begin{align}\label{yihan20}
\int_{\T^3}\la^{s} {F_4}\cdot\la^{s} \b\,dx
\le&\frac{\sigma}{16}\|\b\|_{H^{s+1}}^2\nn\\
&+C(\norm{\nabla \u}{L^\infty}+\norm{\nabla \b}{L^\infty}+\|\b\|_{L^{\infty}}^2)(\norm{\la^{s}\u}{L^2}^2+\norm{\la^{s}\b}{L^2}^2).
\end{align}
We now bound the term involving ${F_2}$ in \eqref{ed4}. To do so, we  write
\begin{align}\label{yihan21}
\int_{\T^3}\la^{s} {F_2}\cdot\la^{s} \u\,dx=\sum_{i=3}^{8}A_i
\end{align}
with
\begin{align*}
&A_3\stackrel{\mathrm{def}}{=}-\int_{\T^3}\la^{s} (\u\cdot\nabla \u)\cdot\la^{s} \u\,dx,\ \qquad
A_4\stackrel{\mathrm{def}}{=}\int_{\T^3}\la^{s} (\b\cdot\nabla\b)\cdot\la^{s}
\u\,dx,\nn\\
&A_5\stackrel{\mathrm{def}}{=}-\int_{\T^3}\la^{s} \left(\frac{\ta-a}{1+a}\right)\cdot\la^{s} \u\,dx,\qquad
A_6\stackrel{\mathrm{def}}{=}\int_{\T^3}\la^{s} (I(a)({\n}\cdot\nabla
\b-{\n}\nabla \b))\cdot\la^{s} \u\,dx,\nn\\
&A_7\stackrel{\mathrm{def}}{=}\int_{\T^3}\la^{s} (I(a)(\b\cdot\nabla\b-\b\nabla\b))\cdot\la^{s} \u\,dx.\nn
\end{align*}
As in \eqref{ed5-1},
\begin{align*}
A_3\le C\norm{\nabla \u}{L^\infty}\norm{\la^{s}\u}{L^2}^2.
\end{align*}
In view of $\div\b=0,$ one can write
\begin{align*}
A_4=&\int_{\T^3}\la^{s} \div(\b\otimes\b)\cdot\la^{s} \u\,dx
\leq C\|\b\|_{L^{\infty}}\|\u\|_{H^{s}}
        \norm{\la^{s+1}\b}{L^2}\nonumber\\
         \leq&\frac{\sigma}{16}\|\b\|_{H^{s+1}}^2+C\|\b\|_{L^{\infty}}^2
         \|\u\|_{H^{s}}.
\end{align*}
To bound $A_5$, we write
$$\G :=\frac{\ta-a}{1+a}.
$$
Then
\begin{align*}
A_5=&-
\int_{\T^3}\la^{s} \left(\frac{\ta-a}{1+a}\nabla a\right)\cdot\la^{s} \u\,dx
=-\int_{\T^3}\la^{s} \left(\G\nabla a\right)\cdot\la^{s} \u\,dx\nn\\
=&-\sum_{1\le s\le \ell-1}\int_{\T^3}\la^{s} \left(\G\nabla a\right)\cdot\la^{s} \u\,dx
-\sum_{s= \ell}\int_{\T^3}\la^{s} \left(\G\nabla a\right)\cdot\la^{s} \u\,dx\nn\\
:=& \, A_{5,1}+A_{5,2}.
\end{align*}
$A_{51}$ can be bounded directly via Lemma \ref{daishu},
\begin{align*}
A_{5,1}=&-\sum_{1\le s\le \ell-1}\int_{\T^3}\la^{s} \left(\G\nabla a\right)\cdot\la^{s} \u\,dx\nn\\
\le & \,C(\|\G\|_{L^\infty}\|\nabla a\|_{H^{\ell-1}}+\|\nabla
a\|_{L^\infty}\|a\|_{H^{\ell-1}})\|\u\|_{H^{\ell}}\nn\\
\le &\, C(\|(a,\ta)\|_{L^\infty}+\|\nabla
a\|_{L^\infty})\|(a,\u)\|_{H^{\ell}}^2.
\end{align*}
The estimates of $A_{5,2}$ is more elaborate and the aim is to avoid $(\ell
+1)$th derivative on $a$.
\begin{align*}
A_{5,2}=&-\int_{\T^3}\la^{\ell} \left(\G\nabla a\right)\cdot\la^{\ell}
\u\,dx\nn\\
=&-\left(\sum_{0\le \alpha\le {\ell-1}}\int_{\T^3}\la^{\ell-\alpha}
\G\la^{\alpha}\nabla a\cdot\la^{\ell} a\,dx\right)-\int_{\T^3}
\G\la^{\ell}\nabla
a\cdot\la^{\ell} \u\,dx\nn\\
=&-\sum_{0\le \alpha\le {\ell-1}}\int_{\T^3}\la^{\ell-\alpha} \G\la^{\alpha}\nabla a\cdot\la^{\ell} a\,dx-\int_{\T^3} \G\la^{\ell}\nabla a\cdot\la^{\ell} \u\,dx\nn\\
=&A_{5,2}^{(1)}+A_{5,2}^{(2)}.
\end{align*}
By Lemma \ref{daishu} and interpolation inequalities,
\begin{align*}
A_{5,2}^{(1)}=&-\sum_{0\le \alpha\le {\ell-1}}\int_{\T^3}\la^{\ell-\alpha} \G\la^{\alpha}\nabla a\cdot\la^{\ell} a\,dx\nn\\
\le& \,C\left(\norm{\G}{L^\infty}\norm{\la^{\ell} a}{L^2}^2
+\norm{\nabla a}{L^\infty}\norm{\la^{\ell-1} \G}{L^2}\norm{\la^{\ell}
a}{L^2}\right)\nn\\
\le&\, C\norm{(a,\ta)}{L^\infty}\norm{a}{H^\ell}^2
+ \,C\norm{\nabla a}{L^\infty}(\norm{a}{H^\ell}^2+\norm{\ta}{H^\ell}^2).
\end{align*}
By integration by parts,
\begin{align}\label{yihan26}
A_{5,2}^{(2)}=&-\int_{\T^3} \G\la^{\ell}\nabla a\cdot\la^{\ell} \u\,dx\nn\\
=&\int_{\T^3} \nabla\G\la^{\ell} a\cdot\la^{\ell} \u\,dx+\int_{\T^3}
\G\la^{\ell} a\cdot\la^{\ell} \div\u\,dx.
\end{align}
Clearly,
\begin{align*}
\int_{\T^3} \nabla\G\la^{\ell} a\cdot\la^{\ell} \u\,dx
\le& \,C\|\nabla \G\|_{L^{\infty}}(\|a\|_{H^{\ell}}^2+\|\u\|_{H^{\ell}}^2)\nn\\
\le& \,C(\|\nabla a\|_{L^{\infty}}+\|\nabla
\ta\|_{L^{\infty}})(\|a\|_{H^{\ell}}^2+\|\u\|_{H^{\ell}}^2).
\end{align*}
To bound the second term in \eqref{yihan26}, we invoke
 $$\div\u=-\frac{\partial_t a+\u\cdot\nabla a}{1+a}$$
 to obtain
\begin{align*}
\int_{\T^3}\G\la^{\ell} \div\u\cdot\la^{\ell} a\,dx
=&-\int_{\T^3}\G\la^{\ell} \left(\frac{\partial_t a+\u\cdot\nabla a}{1+a}\right)\cdot\la^{\ell} a\,dx \nn\\
=&-\int_{\T^3}\G\la^{\ell} \left(\frac{\partial_t a}{1+a}\right)\cdot\la^{\ell}
a\,dx -\int_{\T^3}\G\la^{\ell} \left(\frac{\u\cdot\nabla
a}{1+a}\right)\cdot\la^{\ell} a\,dx\nn\\
=&\, H_1+H_2.
\end{align*}
By the product rule,
\begin{align*}
H_1=&-\int_{\T^3}\G\la^{\ell} \left(\frac{\partial_t a}{1+a}\right)\cdot\la^{\ell} a\,dx\nn\\
=&-\int_{\T^3}\frac{\G}{1+a}\la^{\ell} \left({\partial_t a}\right)\cdot\la^{\ell} a\,dx
-\int_{\T^3}\G\partial_t a\la^{\ell}\left(\frac{1}{1+a}\right)\cdot\la^{\ell} a\,dx\nn\\
=&-\frac12\int_{\T^3}\frac{\G}{1+a}\partial_t(\la^{\ell}a)^2 \,dx
-\int_{\T^3}\G\partial_t a\la^{\ell}\left(\frac{1}{1+a}\right)\cdot\la^{\ell}
a\,dx\nn\\
=&-\frac12\frac{d}{dt}\int_{\T^3}\frac{\G}{1+a}(\la^{\ell}a)^2 \,dx
+\frac12\int_{\T^3}\partial_t \left(\frac{\G}{1+a}\right)(\la^{\ell}a)^2\,dx
\nn\\
& - \int_{\T^3}\G\partial_t
a\la^{\ell}\left(\frac{1}{1+a}\right)\cdot\la^{\ell}
a\,dx\nn\\
=&-\frac12\frac{d}{dt}\int_{\T^3}\frac{\G}{1+a}(\la^{\ell}a)^2 \,dx
+\frac12\int_{\T^3}\frac{\G}{(1+a)^2}\partial_t a(\la^{\ell}a)^2\,dx
\nn\\
& +\frac12\int_{\T^3}
\left(\frac{\partial_t\G}{1+a}\right)(\la^{\ell}a)^2\,dx -
\int_{\T^3}\G\partial_t
a\la^{\ell}\left(\frac{1}{1+a}\right)\cdot\la^{\ell}
a\,dx\nn\\
=&H_{1,1}+H_{1,2}+H_{1,3} + H_{1,4}.
\end{align*}
To estimate the terms on the right, we invoke the following simple bounds,
\begin{align*}
\norma{\frac{\G}{(1+a)^2}}{L^\infty}
=&\norma{\frac{\ta-a}{(1+a)^3}}{L^\infty}\le \,  C(
\norm{a}{L^\infty}+\norm{\ta}{L^\infty})
\end{align*}
and
\begin{align*}
\norm{\partial_ta}{L^\infty}
=&\norm{-\u\cdot\nabla a-a\div \u-\div \u}{L^\infty}\nn\\
\le& C((1+
\norm{a}{L^\infty})\norm{\nabla \u}{L^\infty}+\norm{\nabla a}{L^\infty}\norm{ \u}{L^\infty}).
\end{align*}
By the definition of $\G$,
 \begin{align*}
 \partial_t\G
 =\frac{\partial_t\ta}{1+a}-\frac{1+\ta}{(1+a)^2}\partial_t a.
 \end{align*}
Invoking the equation of $\ta$,
 \begin{align*}
\norm{\partial_t\ta}{L^\infty}
=&\norma{-\u\cdot\nabla\ta-\div\u- \div(\theta \u)+\kappa\Delta\ta -\kappa I(a)\Delta\ta+\frac{|\nabla\times\b|^2}{1+a}}{L^\infty}\nn\\
\le& C((1+
\norm{\ta}{L^\infty})\norm{\nabla \u}{L^\infty}+\norm{\nabla \ta}{L^\infty}\norm{ \u}{L^\infty})\nn\\
&+C(1+
\norm{a}{L^\infty})\norm{\Delta \ta}{L^\infty}+\norm{\nabla \b}{L^\infty}^2.
\end{align*}
Therefore,
\begin{align*}
\norma{\frac{\partial_t\G}{1+a}}{L^\infty}
\le& \,C((1+
\norm{\ta}{L^\infty})\norm{\nabla \u}{L^\infty}+(\norm{\nabla \ta}{L^\infty} +
\norm{\nabla a}{L^\infty})\norm{ \u}{L^\infty})\nn\\
&+C(1+
\norm{a}{L^\infty})\norm{\Delta \ta}{L^\infty}+\norm{\nabla \b}{L^\infty}^2.
\end{align*}
As a consequence,
\begin{align}\label{yihan35}
H_{1,2}+H_{1,3}=&\frac12\int_{\T^3}\frac{\G}{(1+a)^2}\partial_t a(\la^{\ell}a)^2\,dx
+\frac12\int_{\T^3} \left(\frac{\partial_t\G}{1+a}\right)(\la^{\ell}a)^2\,dx\nn\\
\le& C(
\norm{a}{L^\infty}+\norm{\ta}{L^\infty})((1+
\norm{a}{L^\infty})\norm{\nabla \u}{L^\infty}+\norm{\nabla a}{L^\infty}\norm{ \u}{L^\infty})
\norm{a}{H^{\ell}}^2
\nn\\
&+ C((1+
\norm{a}{L^\infty})\norm{\nabla \u}{L^\infty}+\norm{\nabla a}{L^\infty}\norm{ \u}{L^\infty})\norm{a}{H^{\ell}}^2\nn\\
  &+C((1+
\norm{\ta}{L^\infty})\norm{\nabla \u}{L^\infty}+\norm{\nabla \ta}{L^\infty}\norm{ \u}{L^\infty})\norm{a}{H^{\ell}}^2\nn\\
&+C\left((1+
\norm{a}{L^\infty})\norm{\Delta \ta}{L^\infty}+\norm{\nabla \b}{L^\infty}^2\right)\norm{a}{H^{\ell}}^2.
\end{align}
The estimate of $H_{1,4}$ is direct,
\begin{align*}
	H_{1,4} \le \|\G\|_{L^\infty} \norm{\partial_t\ta}{L^\infty}
	\norma{\la^{\ell}\left(\frac{1}{1+a}\right)}{L^2}\,
	\norm{\la^{\ell}a}{L^2},
\end{align*}
which is certainly majorized by the upper bound in \eqref{yihan35}. Therefore,
\begin{align*}
H_1
\le&-\frac12\frac{d}{dt}\int_{\T^3}\frac{\G}{1+a}(\la^{\ell}a)^2 \,dx+CY_\infty(t)\norm{a}{H^{\ell}}^2.
\end{align*}
To estimate $H_2$, we further divide it into two terms,
\begin{align*}
H_2=&-\int_{\T^3}\G\la^{\ell} \left(\frac{\u\cdot\nabla a}{1+a}\right)\cdot\la^{\ell} a\,dx\nn\\
=&-\int_{\T^3}\frac{\G}{1+a}\la^{\ell}(\u\cdot\nabla a)\cdot\la^{\ell} a\,dx-
\int_{\T^3}\frac{\G}{1+a}(\u\cdot\nabla a)\la^{\ell}
\left(\frac1{1+a}\right)\cdot\la^{\ell} a\,dx\nn\\
=&H_{2,1}+H_{2,2}.
\end{align*}
The idea is still to avoid $(\ell +1)$th derivative on $a$ in $H_{2,1}$. To
serve this purpose, we use a commutator and write
\begin{align*}
H_{2,1}=&-\int_{\T^3}\frac{\G}{1+a}\left(\la^{\ell}(\u\cdot\nabla a)-\u\cdot\nabla \la^{\ell} a\right)\cdot\la^{\ell} a\,dx+
\int_{\T^3}\frac{\G}{1+a}\u\cdot\nabla\la^{\ell} a\cdot\la^{\ell} a\,dx\nn\\
=&H_{2,1}^{(1)}+H_{2,1}^{(2)}.
\end{align*}
By Lemma \ref{jiaohuanzi},
\begin{align*}
H_{2,1}^{(1)}\le&\,
C\norma{\frac{\G}{1+a}}{L^\infty}\norma{[\la^\ell,\u\cdot\nabla
]a}{L^2}\norm{\la^{\ell}a}{L^2}\nn\\
\le& \,C(
\norm{a}{L^\infty}+\norm{\ta}{L^\infty})(\norm{\nabla \u}{L^\infty}\norm{\la^{\ell}a}{L^2}+\norm{\la^{\ell} \u}{L^2}\norm{\nabla a}{L^\infty})\norm{\la^{\ell}a}{L^2}\nn\\
\le& \,C(
\norm{a}{L^\infty}+\norm{\ta}{L^\infty})(\norm{\nabla \u}{L^\infty}+\norm{\nabla a}{L^\infty})\norm{(a,\u)}{H^{\ell}}^2.
\end{align*}
By integration by parts,
\begin{align*}
H_{2,1}^{(2)}\le& \,
C\norma{\div\left(\frac{\G\u}{1+a}\right)}{L^\infty}\norm{\la^{\ell}a}{L^2}^2\nn\\
\le&C(\norm{\nabla \u}{L^\infty}+\norm{ \u}{L^\infty}(\norm{\nabla a}{L^\infty}+\norm{\nabla \ta}{L^\infty}))\norm{a}{H^{\ell}}^2.
\end{align*}
By Lemma \ref{fuhe}, $H_{2,2}$ can be bounded by
\begin{align*}
	H_{2,2} &\le \norma{\frac{\G\u}{1+a}}{L^\infty} \, \norm{
	\u}{L^\infty}\,\norm{\nabla a}{L^\infty} \norma{\la^{\ell}
		\left(\frac1{1+a}\right)}{L^2} \norm{\la^{\ell} a}{L^2} \\
		 &\le C\, \left(\norm{a}{L^\infty} +  \norm{\ta}{L^\infty}\right)
		 \norm{
		 	\u}{L^\infty}\,\norm{\nabla a}{L^\infty} \norm{a}{H^{\ell}}^2.
\end{align*}	
Therefore,
\begin{align*}
H_{2}
\le& \, C Y_\infty(t)\norm{(a,\u)}{H^{\ell}}^2
\end{align*}
and
\begin{align*}
\int_{\T^3}\G\la^{\ell} \div\u\cdot\la^{\ell} a\,dx
\le&-\frac12\frac{d}{dt}\int_{\T^3}\frac{\G}{1+a}(\la^{\ell}a)^2
\,dx+CY_\infty(t)\norm{(a,\u)}{H^{\ell}}^2.
\end{align*}
This leads to the following upper bound on $A_5$,
\begin{align*}
A_5\le&-\frac12\frac{d}{dt}\int_{\T^3}\frac{\G}{1+a}(\la^{\ell}a)^2 \,dx+CY_\infty(t)\norm{(a,\u)}{H^{\ell}}^2.
\end{align*}
By Lemma \ref{daishu},
\begin{align*}
A_6 \leq&C\big(\|I(a)\|_{L^{\infty}}\|\nabla \b\|_{H^{s}}
        +\|I(a)\|_{H^{s}}\|\nabla \b\|_{L^{\infty}}\big)\norm{\la^{s}\u}{L^2}\nonumber\\
         \leq&\frac{\sigma}{16}\norm{\la^{s+1}\b}{L^2}^2
         +C(\|a\|_{L^{\infty}}^2\norm{\la^{s}\u}{L^2}^2
        +\|\nabla
        \b\|_{L^{\infty}}(\norm{\la^{s}a}{L^2}^2+\norm{\la^{s}\u}{L^2}^2)).
\end{align*}
By Lemmas \ref{daishu} and \ref{fuhe},
\begin{align*}
A_7 \leq& \, C\big(\|I(a)\|_{L^{\infty}}\|\b\nabla \b\|_{H^{s}}
        + \|I(a)\|_{H^{s}}\|\b\nabla
        \b\|_{L^{\infty}}\big)\norm{\la^{s}\u}{L^2}\nn\\
\le &\,C\big(\|I(a)\|_{L^{\infty}}
(\|\b\|_{L^{\infty}}\| \nabla\b\|_{H^{s}}+\|\nabla\b\|_{L^{\infty}}\|
\b\|_{H^{s}}) + \|I(a)\|_{H^{s}}\|\b\nabla
\b\|_{L^{\infty}}\big)\norm{\la^{s}\u}{L^2}\\
\leq&\frac{\sigma}{16}\norm{\la^{s+1}\b}{L^2}^2
+C\| a\|_{L^{\infty}}\|\nabla \b\|_{L^{\infty}}(\norm{\la^{s}\b}{L^2}^2+\norm{\la^{s}\u}{L^2}^2)
\nn\\
& + \,C(\|a\|_{L^{\infty}}^2\|\b\|_{L^{\infty}}^2\norm{\la^{s}\u}{L^2}^2
        +\| \b\|_{L^{\infty}}\|\nabla
        \b\|_{L^{\infty}}(\norm{\la^{s}a}{L^2}^2+\norm{\la^{s}\u}{L^2}^2)).
\end{align*}
Inserting the bounds for $A_3$ through $A_7$ in \eqref{yihan21} yields
\begin{align}\label{ed15}
\int_{\T^3}\la^{s} {F_2}\cdot\la^{s} \u\,dx
\leq&\frac{\sigma}{8}\norm{\nabla
\b}{H^{\ell}}^2-\frac12\frac{d}{dt}\int_{\T^3}\frac{\G}{1+a}(\la^{\ell}a)^2
\,dx+CY_\infty(t)\norm{(a,\u,\b)}{H^{\ell}}^2.
\end{align}
Finally we bound the term involving ${F_3}$ in \eqref{ed4},
\begin{align*}
\int_{\T^3}\la^{s} {F_3}\cdot\la^{s} \ta\,dx=A_{8}+A_{9}+A_{10},
\end{align*}
where
\begin{align*}
&A_{8}\stackrel{\mathrm{def}}{=}-\int_{\T^3}\la^{s} (\div(\theta
\u))\cdot\la^{s} \ta\,dx,\qquad
A_{9}\stackrel{\mathrm{def}}{=}-\kappa\int_{\T^3}\la^{s} ((\nabla
I(a))\nabla\ta)\cdot\la^{s} \ta\,dx,\nn\\
&A_{10}\stackrel{\mathrm{def}}{=}\int_{\T^3}\la^{s}
\Big(\frac{|\nabla\times\b|^2}{1+a} \Big)\cdot\la^{s} \ta\,dx.
\end{align*}
$A_{8}$ can be bounded similarly as in \eqref{yihan20},
\begin{align}\label{taw2}
A_{8}
\le&\frac{\kappa}{16}\|\la^{s+1}\ta\|_{L^2}^2+C(\norm{
\u}{L^\infty}^2+\|\ta\|_{L^{\infty}}^2)(\norm{\la^{s}\ta}{L^2}^2+\norm{\la^{s}\u}{L^2}^2).
\end{align}
As in $A_6$ and $A_7$,
\begin{align}\label{taw3}
A_{9}\leq&\frac{\kappa}{16}\norm{\la^{s+1}\ta}{L^2}^2+C(\|\nabla a\|_{L^{\infty}}^2\norm{\la^{s}\ta}{L^2}^2
        +\|\nabla \ta\|_{L^{\infty}}^2\norm{\la^{s}a}{L^2}^2).
\end{align}
To bound $A_{10}$, we first rewrite it as
\begin{align}\label{taw4}
A_{10} \leq&\int_{\T^3}\la^{s} \Big((1-I(a))|\nabla\times\b|^2\Big)\cdot\la^{s}
\ta\,dx\nn\\
= &\int_{\T^3}\la^{s} \Big(|\nabla\times\b|^2\Big)\cdot\la^{s}
\ta\,dx -\int_{\T^3}\la^{s} \Big(I(a)|\nabla\times\b|^2\Big)\cdot\la^{s}
\ta\,dx\nn\\
:=&A_{10}^{(1)}+A_{10}^{(2)}.
\end{align}
By Lemma \ref{daishu},
\begin{align*}
A_{10}^{(1)}\leq&C\big(\|\nabla\b\|_{L^{\infty}}\|\nabla\b\|_{H^{s-1}}
        \big)\norm{\la^{s+1}\ta}{L^2}\nn\\
         \leq&\frac{\kappa}{16}\|\la^{s+1}\ta\|_{L^2}^2+C\|\nabla\b\|_{L^{\infty}}^2
         \norm{\la^{s}\b}{L^2}^2,
\end{align*}
and
\begin{align*}
A_{10}^{(2)}\leq&\,C \left(\|I(a)\|_{L^{\infty}}\||\nabla\b|^2\|_{H^{s-1}}
+\||\nabla\b|^2\|_{L^{\infty}}\|I(a)\|_{H^{s-1}}\right)
        \norm{\la^{s+1}\ta}{L^2}\nn\\
         \leq&\frac{\kappa}{16}\|\la^{s+1}\ta\|_{L^2}^2
         +C\big(\|a\|_{L^{\infty}}^2\|\nabla\b\|_{L^{\infty}}^2\|\b\|_{H^{s}}^2
+\|\nabla\b\|_{L^{\infty}}^4\|a\|_{H^{s}}^2
        \big).
\end{align*}
Inserting the above two estimates in \eqref{taw4}, we obtain
\begin{align}\label{taw7}
A_{10}\leq& \, \frac{\kappa}{8}\|\la^{s+1}\ta\|_{L^2}^2
         +C(1+\|a\|_{L^{\infty}}^2+\|\nabla\b\|_{L^{\infty}}^2)
         \|\nabla\b\|_{L^{\infty}}^2(\|\b\|_{H^{s}}^2+\|a\|_{H^{s}}^2).
\end{align}
Combining \eqref{taw2}, \eqref{taw3} and \eqref{taw7} leads to
\begin{align}\label{taw8}
\int_{\T^3}\la^{s} {F_3}\cdot\la^{s} \ta\,dx
\le&
\,\frac{3\kappa}{16}\norm{\la^{s+1}\ta}{L^2}^2+C(1+\|a\|_{L^{\infty}}^2+\|\nabla\b\|_{L^{\infty}}^2)
         \|\nabla\b\|_{L^{\infty}}^2\|(a,\b)\|_{H^{s}}^2\nn\\
&+C(\norm{(\nabla \b,\nabla \ta)}{L^\infty}+\|(\u,\ta,\nabla a,\nabla\ta)\|_{L^{\infty}}^2)\|(a,\u,\ta)\|_{H^{s}}^2.
\end{align}
Inserting \eqref{yihan7}, \eqref{yihan20}, \eqref{ed15}, and  \eqref{taw8} in
\eqref{ed4} and summing up \eqref{basic}, we obtain
\begin{align*}
&\frac12\frac{d}{dt}\norm{(\la^sa,\la^s\u,\la^s\ta,\la^s\b)}{L^2}^2
+\frac12\frac{d}{dt}\int_{\T^3}\frac{a}{1+a}(\la^{\ell}a)^2 \,dx
+\frac12\frac{d}{dt}\int_{\T^3}\frac{\ta-a}{(1+a)^2}(\la^{\ell}a)^2 \,dx\nn\\
&\qquad+ \sigma\norm{\la^{\ell+1}\b}{L^2}^2+\frac12
{c_0^{-1}}\kappa\norm{\la^{\ell+1}\ta}{L^2}^2\nn\\
&\quad\le CY_\infty(t)\|(a,\u,\ta,\b)\|_{H^{\ell}}^2,
\end{align*}
which implies that
\begin{align*}
&\frac12\frac{d}{dt}\left(\|(a,\u,\ta,\b)\|_{H^{\ell}}^2
+\int_{\T^3}\frac{\ta+a^2}{(1+a)^2}(\la^{\ell}a)^2 \,dx\right)+
\sigma\norm{\la^{\ell+1}\b}{L^2}^2+\frac12
{c_0^{-1}}\kappa\norm{\la^{\ell+1}\ta}{L^2}^2\nn\\
&\quad\le CY_\infty(t)\|(a,\u,\ta,\b)\|_{H^{\ell}}^2.
\end{align*}
This completes the proof of Proposition \ref{high2}.
\end{proof}

\vskip .3in
\section{ The dissipation of $a$}
\label{sec:adiss}

The equation of $a$ doesn't involve any damping or dissipation, but the
coupling and interaction between the equation of $a$ and that of $\u$ actually
generates some weak dissipation and stabilizing effect. Mathematically there is
a wave structure in the equations of $a$ and $\div\u$. In fact, the
linearized equations
of $a$ and $\div \u$ are given by
\begin{align*}
	& \partial_t a + \div \u =0,\\
	& \partial_t \div \u+ R\Delta a+ R\Delta \ta
	=-\Delta ({\n}\cdot\b),
\end{align*}
which can be converted into the following wave equations
\begin{align*}
	& \partial_{tt} a - R \Delta a = R \Delta \theta - \Delta({\n}\cdot\b) ,\\
	& \partial_{tt} \div \u - R\Delta\div \u =- R\Delta \partial_t\ta
	-\Delta ({\n}\cdot\partial_t\b).
\end{align*}
Making use of this structure by constructing suitable Lyapunov functional, we
are able to prove the following proposition.

\begin{proposition} \label{adis}
Let ${N}\ge 4r+7$ with $r>2$. Assume the solution $(a(t), \u(t), \ta(t),
\b(t))$ to \eqref{m2} satisfies
	\begin{align}\label{shiping2}
		\sup_{t\in[0,T]}\norm{(a(t), \u(t), \ta(t), \b(t))}{H^{N}}\le \delta
	\end{align}
	for some $0<\delta<1.$ Then
\begin{align}\label{yiming1}
	&\norm{\nabla a}{H^{{r+3}}}^2+\sum_{0\le s\le r+3}\frac{d}{dt}{\big\langle{\Lambda^{s}} \u,{\Lambda^{s}}\nabla a\big\rangle}\nn\\
	&\quad\le C\norm{\div\u}{H^{{r+3}}}^2+
C(1+\delta^2)\norm{ (\ta,\b)}{H^{r+4}}^2+ C\delta^2 \norm{{\n}\cdot\nabla \u}{H^{r+3}}^2.
\end{align}
\end{proposition}

\begin{proof}
It follows from the equation
$$\nabla a
=-\partial_t \u-\nabla \ta+{\n}\cdot\nabla \b-\nabla ({\n}\cdot\b)+f_2
$$
that
\begin{align}\label{yiming2}
\norm{{\Lambda^{s}}\nabla a}{L^2}^2
=&-{\big\langle{\Lambda^{s}}\partial_t \u,{\Lambda^{s}}((\nabla a))\big\rangle}-{\big\langle{\Lambda^{s}} \nabla\ta,{\Lambda^{s}}\nabla a\big\rangle}\nn\\
&+{\big\langle{\Lambda^{s}} (\n\cdot\nabla\b),{\Lambda^{s}}(\nabla
a)\big\rangle}-{\big\langle{\Lambda^{s}} (\n\nabla\b),{\Lambda^{s}}(\nabla
a)\big\rangle}\nn\\
&+{\big\langle{\Lambda^{s}} (f_2),{\Lambda^{s}}\nabla a\big\rangle}\nn\\
=&M_{1}+M_{2}+M_{3}+M_{4}+M_{5}.
\end{align}
$M_1$ can be rewritten as
\begin{align}\label{yiming3}
M_{1}=&-{\big\langle{\Lambda^{s}}\partial_t \u,{\Lambda^{s}}((\nabla
a))\big\rangle}\nn\\
=&-\frac{d}{dt}{\big\langle{\Lambda^{s}} \u,{\Lambda^{s}}\nabla a\big\rangle}-{\big\langle{\Lambda^{s}}\div\u,{\Lambda^{s}} \partial_t a\big\rangle}\nn\\
=&-\frac{d}{dt}{\big\langle{\Lambda^{s}} \u,{\Lambda^{s}}\nabla a\big\rangle}
+{\big\langle{\Lambda^{s}}\div\u,{\Lambda^{s}}\div\u \big\rangle}-{\big\langle{\Lambda^{s}}\div\u,{\Lambda^{s}} f_1\big\rangle}\nn\\
=& -\frac{d}{dt}{\big\langle{\Lambda^{s}} \u,{\Lambda^{s}}\nabla a\big\rangle}+\norm{{\Lambda^{s}}(\div\u)}{L^2}^2-{\big\langle{\Lambda^{s}}\div\u,{\Lambda^{s}} f_1\big\rangle}.
\end{align}
Recall that
$$
f_1{=}-\u\cdot\nabla a-a\div \u.
$$
We assume that $0\le s\le r+3$. The last term of $M_{1}$ in \eqref{yiming3} can
be bounded as
\begin{align}\label{yiming15}
	-{\big\langle{\Lambda^{s}}\div\u,{\Lambda^{s}} f_1\big\rangle}
	\le& \, (1+ \|a\|_{L^\infty}) \, \norm{\div\u}{H^{{r+3}}}^2 + C\,
	\norm{\u}{H^N}^2\norm{ a}{H^{r+4}}^2 \nn\\
	\le& \, C\norm{\div\u}{H^{{r+3}}}^2+C\delta^2\norm{a}{H^{r+4}}^2,
\end{align}
where we have used (\ref{shiping2}). By H\"{o}lder's inequality,
\begin{align}\label{yiming4}
M_{2}+M_{3}+M_{4}\le\frac1{16}\norm{\Lambda^{r+3}\nabla
a}{L^2}^2+C\norm{\nabla\ta}{H^{r+3}}^2+C\norm{\nabla\b}{H^{r+3}}^2
\end{align}
and
\begin{align}\label{yiming6}
	M_{5}=	{\big\langle{\Lambda^{s}} (f_2),{\Lambda^{s}}\nabla a\big\rangle}
	\le\frac1{16}\norm{{\Lambda^{s}}\nabla
	a}{L^2}^2+\norm{{\Lambda^{s}}f_2}{L^2}^2.
\end{align}
Recall that
\begin{align*}
f_2 := &-\u\cdot\nabla \u+\b\cdot\nabla\b+\b\nabla\b+I(a)\nabla a-\theta\nabla
J(a)
\nn\\
&-I(a)({\n}\cdot\nabla \b+\b\cdot\nabla\b-{\n}\nabla \b-\b\nabla\b).
\end{align*}
We now deal with the terms in $f_2$.
By Lemma \ref{daishu}, (\ref{shiping2}) and Lemma \ref{dir},
\begin{align}\label{yiming7}
\norm{\la^{s}  (\u\cdot\nabla \u)}{L^{2}}^2
\le& C(\norm{\u}{L^\infty}^2\norm{\nabla \u}{H^s}^2+\norm{\nabla \u}{L^\infty}^2\norm{ \u}{H^s}^2 )\nn\\
\le& C(\norm{\u}{H^2}^2\norm{ \u}{H^{s+1}}^2+\norm{\nabla \u}{H^2}^2\norm{ \u}{H^s}^2 )\nn\\
\le&C\norm{ \u}{H^{s+1}}^2\norm{\u}{H^3}^2
\nn\\
\le& C\delta^2\norm{{\n}\cdot\nabla \u}{H^{r+3}}^2.
\end{align}
Similarly,
\begin{align}\label{yiming8}
\norm{\la^{s}  (\b\cdot\nabla \b+\b\nabla\b)}{L^{2}}^2
\le& C(\norm{\b}{L^\infty}^2\norm{\nabla \b}{H^s}^2+\norm{\nabla \b}{L^\infty}^2\norm{ \b}{H^s}^2 )\nn\\
\le& C(\norm{\b}{H^2}^2\norm{ \b}{H^{s+1}}^2+\norm{\nabla \b}{H^2}^2\norm{ \b}{H^s}^2 )\nn\\
\le&C\norm{ \b}{H^{s+1}}^2\norm{\b}{H^3}^2
\nn\\
\le& C\delta^2\norm{\b}{H^{r+4}}^2,
\end{align}
\begin{align}\label{yiming9}
\norm{\la^{s}  (\ta\nabla J(a) )}{L^{2}}^2
\le& C(\norm{\ta}{L^\infty}^2\norm{\nabla J(a)}{H^s}^2+\norm{\nabla J(a)}{L^\infty}^2\norm{ \ta}{H^s}^2 )\nn\\
\le& C(\norm{\ta}{H^2}^2\norm{ a}{H^{s+1}}^2+\norm{a}{H^3}^2\norm{\ta}{H^s}^2 )\nn\\
\le&C\delta^2\norm{ \ta}{H^{r+4}}^2,
\end{align}
\begin{align}\label{yiming10}
\norm{\la^{s}  (I(a)\nabla a )}{L^{2}}^2
\le& C(\norm{I(a)}{L^\infty}^2\norm{\nabla a}{H^s}^2+\norm{\nabla a}{L^\infty}^2\norm{ I(a)}{H^s}^2 )\nn\\
\le& C\norm{a}{H^N}^2\norm{a}{H^{r+4}}^2 \nn\\
\le&C\delta^2\norm{a}{H^{r+4}}^2,
\end{align}
\begin{align}\label{yiming11}
\norm{\la^{s}  (I(a){\n}\nabla \b )}{L^{2}}^2
\le& C(\norm{I(a)}{L^\infty}^2\norm{{\n}\nabla \b}{H^s}^2+\norm{{\n}\nabla \b}{L^\infty}^2\norm{ I(a)}{H^s}^2 )\nn\\
\le& C(\norm{a}{H^3}^2\norm{ {\n}\nabla \b}{H^{s}}^2+\norm{\b}{H^3}^2\norm{
a}{H^s}^2 )\nn\\
\le& C\delta^2\norm{ \b}{H^{r+4}}^2
\end{align}
and
\begin{align}\label{yiming12}
&\norm{\la^{s}  (I(a)({\n}\cdot\nabla \b)}{L^{2}}^2+\norm{\la^{s}  (I(a)({\b}\cdot\nabla \b-{\b}\nabla \b)}{L^{2}}^2
 \le C\delta^2\|\b\|_{H^{r+4}}^2.
\end{align}
Making use of \eqref{yiming7} through \eqref{yiming12} gives
\begin{align}\label{yiming13}
\norm{\la^{s}  f_2}{L^{2}}^2\le& C\delta^2\norm{ (a,\ta,\b)}{H^{r+4}}^2+ C\delta^2 \norm{{\n}\cdot\nabla \u}{H^{r+3}}^2.
\end{align}
Combining \eqref{yiming2}, \eqref{yiming3}, \eqref{yiming15}, \eqref{yiming4},
\eqref{yiming6} and \eqref{yiming13} leads to \eqref{yiming1}. This completes
the proof of Proposition \ref{adis}.
\end{proof}

\vskip .3in
\section{ The dissipation of $\n\cdot\nabla \u$}
\label{sec:nudiss}

This section rigorously establishes the stabilizing effect of the background
magnetic field. The velocity equation satisfies the Euler equation which
involves no dissipation. This section proves a proposition that demonstrates
the dissipative effect of the velocity field.

\vskip .1in
Mathematically the interaction of the fluid and the magnetic field
near a background magnetic field generates a wave structure. For the sake of
simplicity, we consider the linearized system of $\u$ and $\b$,
\begin{align*}
	&\partial_t \u+ R\nabla a+ R\nabla \ta
	={\n}\cdot\nabla \b-\nabla ({\n}\cdot\b),\\
	&\partial_t \b-\sigma\Delta\b={\n}\cdot\nabla \u-\n\div\u.
\end{align*}
To get to the point, we further ignore $R\nabla a+ R\nabla \ta$. Then we obtain
the following degenerate wave equations
\begin{align*}
&		\partial_{tt} \u - \sigma \Delta \partial_t \u - ({\n}\cdot\nabla)^2 \u
=-\nabla((\n\otimes \n)\cdot \nabla \u) +\nabla\div \u- (\n\cdot
\nabla\div\u)\n, \\
&	\partial_{tt} \b - \sigma \Delta \partial_t \b - ({\n}\cdot\nabla)^2 \b
=-\nabla((\n\otimes \n)\cdot \nabla \b) +\n \Delta
({\n}\cdot\b).
\end{align*}
$\u$ and $\b$ share a very similar wave structure. In comparison with the
original equation of $u$, the wave equation contains two extra regularizing
terms. $- \sigma \Delta \p_t \u$ comes from the magnetic diffusion and $-
({\n}\cdot\nabla)^2 \u$ is due to the magnetic field. $- ({\n}\cdot\nabla)^2
\u$ allows us to control the derivative of $\u$ along the background magnetic
field. More precisely, we are able to prove the following proposition.

\begin{proposition}\label{shiping1}
Let ${N}\ge 4r+7$ with $r>2$. Assume the solution $(a(t), \u(t), \ta(t),
\b(t))$ to \eqref{m2} satisfies \eqref{shiping2}. Then
\begin{align}\label{shiping3}
&\norm{{\n}\cdot\nabla \u}{H^{r+3}}^2-\sum_{0\le s\le
r+3}\frac{d}{dt}\big\langle{\Lambda^s}\b,{\Lambda^s}({\n}\cdot\nabla
\u)\big\rangle\le C\norm{ \b}{H^{r+5}}^2+C\norm{
\ta}{H^{r+5}}^2+\varepsilon\norm{ \nabla a}{H^{r+3}}^2,
\end{align}
where $\varepsilon>0$ is a fixed small number.
\end{proposition}

\begin{proof}
Applying ${\Lambda^s}$ with $0\le s\le r+3$ to the fourth equation of
\eqref{m2}, multiplying by ${\Lambda^s}({\n}\cdot\nabla \u)$ and
then integrating over $\T^3$, we obtain
\begin{align}\label{shiping4}
\norm{{\Lambda^s}({\n}\cdot\nabla \u)}{L^2}^2
=&{\big\langle{\Lambda^s}\partial_t \b,{\Lambda^s}({\n}\cdot\nabla \u)\big\rangle}-{\big\langle{\Lambda^s} \Delta\b,{\Lambda^s}({\n}\cdot\nabla \u)\big\rangle}\nn\\
&+{\big\langle{\Lambda^s} (\n\div\u),{\Lambda^s}({\n}\cdot\nabla \u)\big\rangle}+{\big\langle{\Lambda^s} (f_4),{\Lambda^s}({\n}\cdot\nabla \u)\big\rangle}\nn\\
:=&\Pi_{1}+\Pi_{2}+\Pi_{3}+\Pi_{4}.
\end{align}
$I_1$ can be further written as
\begin{align*}
\Pi_{1}=&\big\langle{\Lambda^s}\partial_t \b,{\Lambda^s}({\n}\cdot\nabla
\u)\big\rangle\nn\\
=&\frac{d}{dt}\big\langle{\Lambda^s}\b,{\Lambda^s}({\n}\cdot\nabla \u)\big\rangle-\big\langle{\Lambda^s}\b,{\Lambda^s}({\n}\cdot\nabla \partial_t\u)\big\rangle\nn\\
=&\frac{d}{dt}\big\langle{\Lambda^s}\b,{\Lambda^s}({\n}\cdot\nabla \u)\big\rangle+\big\langle{\Lambda^s}({\n}\cdot\nabla\b),{\Lambda^s} \partial_t\u\big\rangle\nn\\
=&\frac{d}{dt}\big\langle{\Lambda^s}\b,{\Lambda^s}({\n}\cdot\nabla \u)\big\rangle-\big\langle{\Lambda^s}({\n}\cdot\nabla \b),{\Lambda^s}(\nabla a)\big\rangle\\
&-\big\langle{\Lambda^s}({\n}\cdot\nabla \b),{\Lambda^s}(\nabla \ta)\big\rangle+\big\langle{\Lambda^s}({\n}\cdot\nabla \b),{\Lambda^s}({\n}\cdot\nabla \b)\big\rangle\nn\\
&-\big\langle{\Lambda^s}({\n}\cdot\nabla \b),{\Lambda^s}(\nabla ({\n}\cdot\b))\big\rangle
+\big\langle{\Lambda^s}({\n}\cdot\nabla \b),{\Lambda^s}(f_2)\big\rangle
\nn\\
=&\Pi_{1}^{(1)}+\Pi_{1}^{(2)}+\Pi_{1}^{(3)}+\Pi_{1}^{(4)}+\Pi_{1}^{(5)}
+\Pi_{1}^{(6)}.
\end{align*}
For any fixed small number $\varepsilon>0$,
\begin{align*}
\Pi_{1}^{(2)}=&-\big\langle{\Lambda^s}({\n}\cdot\nabla \b),{\Lambda^s}(\nabla
a)\big\rangle\nn\\
\le& C\norm{{\Lambda^{s}}({\n}\cdot\nabla
\b)}{L^2}\norm{{\Lambda^{s+1}}a}{L^2}\nn\\
\le&\varepsilon\norm{\nabla a}{H^{s}}^2+C\norm{\nabla \b}{H^{s}}^2.
\end{align*}
By H\"{o}lder's inequality,
\begin{align*}
\Pi_{1}^{(3)}+\Pi_{1}^{(4)}+\Pi_{1}^{(5)}
\le&C\norm{\nabla \b}{H^s}^2+C\norm{\nabla \ta}{H^s}^2.
\end{align*}
As in the derivation of \eqref{yiming13}, we have
\begin{align*}
\Pi_{1}^{(6)}=&\big\langle{\Lambda^s}({\n}\cdot\nabla
\b),{\Lambda^s}(f_2)\big\rangle
\nn\\
\le& \varepsilon\norm{ \nabla a}{H^{r+3}}^2+ C\delta^2\norm{
(\ta,\b)}{H^{r+4}}^2+ C\delta^2 \norm{{\n}\cdot\nabla \u}{H^{r+3}}^2.
\end{align*}
By H\"{o}lder's inequality,
\begin{align*}
\Pi_2\le& C\norm{{\Lambda^s}\Delta\b}{L^2}\norm{{\Lambda^s}({\n}\cdot\nabla
\u)}{L^2} \le\frac{1}{16}\norm{{\Lambda^s}({\n}\cdot\nabla
\u)}{L^2}^2+C\norm{\b}{H^{r+5}}^2.
\end{align*}
According to the equation of $\theta$, namely
$\partial_t\theta-\kappa\Delta\ta+ \div \u=f_3$,
\begin{align*}
\Pi_{3}=&{\big\langle{\Lambda^s} (\n\div\u),{\Lambda^s}({\n}\cdot\nabla
\u)\big\rangle}\nn\\
=&
\big\langle{\Lambda^s}\n(-\partial_t\theta),{\Lambda^s}({\n}\cdot\nabla \u)\big\rangle
+ \kappa \big\langle{\Lambda^s}\n(\Delta\ta),{\Lambda^s}({\n}\cdot\nabla
\u)\big\rangle+\big\langle{\Lambda^s}\n(f_3),{\Lambda^s}({\n}\cdot\nabla
\u)\big\rangle
\nn\\
=&\Pi_{3,1}+\Pi_{3,2}+\Pi_{3,3}.
\end{align*}
By H\"{o}lder's inequality,
\begin{align*}
\Pi_{3,2}\le&
\,C\norm{{\Lambda^s}\Delta\ta}{L^2}\norm{{\Lambda^s}({\n}\cdot\nabla \u)}{L^2}
\nn\\
\le&\,\frac{1}{16}\norm{{\Lambda^s}({\n}\cdot\nabla
\u)}{L^2}^2+C\norm{\ta}{H^{r+5}}^2.
\end{align*}
$\Pi_{3,1}$ can be written as
\begin{align*}
\Pi_{3,1}=&\big\langle{\Lambda^s}\n(-\partial_t\theta),{\Lambda^s}({\n}\cdot\nabla
 \u)\big\rangle\nn\\
=&-\frac{d}{dt}\big\langle{\Lambda^s}(\n\ta),{\Lambda^s}({\n}\cdot\nabla \u)\big\rangle+\big\langle{\Lambda^s}(\n\ta),{\Lambda^s}({\n}\cdot\nabla \partial_t\u)\big\rangle\nn\\
=&-\frac{d}{dt}\big\langle{\Lambda^s}(\n\ta),{\Lambda^s}({\n}\cdot\nabla \u)\big\rangle+\big\langle{\Lambda^s}({\n}\cdot\nabla(\n\ta)),{\Lambda^s} \partial_t\u\big\rangle
\nn\\
=&\Pi_{3,1}^{(1)}+\big\langle{\Lambda^s}({\n}\cdot\nabla(\n\ta)),{\Lambda^s}
\partial_t\u\big\rangle.
\end{align*}
Invoking the equation of $\u$ leads to
\begin{align*}
\big\langle{\Lambda^s}({\n}\cdot\nabla(\n\ta)),{\Lambda^s} \partial_t\u\big\rangle
=&-\big\langle{\Lambda^s}({\n}\cdot\nabla(\n\ta)),{\Lambda^s}(\nabla a)\big\rangle\\
&-\big\langle{\Lambda^s}({\n}\cdot\nabla(\n\ta)),{\Lambda^s}(\nabla \ta)\big\rangle+\big\langle{\Lambda^s}({\n}\cdot\nabla(\n\ta)),{\Lambda^s}({\n}\cdot\nabla \b)\big\rangle\nn\\
&-\big\langle{\Lambda^s}({\n}\cdot\nabla(\n\ta)),{\Lambda^s}(\nabla ({\n}\cdot\b))\big\rangle
+\big\langle{\Lambda^s}({\n}\cdot\nabla(\n\ta)),{\Lambda^s}(f_2)\big\rangle\nn\\
=&\Pi_{3,1}^{(2)}+\Pi_{3,1}^{(3)}+\Pi_{3,1}^{(4)}+\Pi_{3,1}^{(5)}+\Pi_{3,1}^{(6)}.
\end{align*}
By H\"{o}lder's inequality,
\begin{align*}
\Pi_{3,1}^{(2)}=&-\big\langle{\Lambda^s}({\n}\cdot\nabla(\n\ta)),{\Lambda^s}(\nabla
 a)\big\rangle \le\,
C\norm{{\Lambda^{s}}({\n}\cdot\nabla(\n\ta))}{L^2}\norm{{\Lambda^s}\nabla
a}{L^2}\nn\\
\le& \, \varepsilon\norm{\nabla a}{H^{r+3}}^2 + \,C\norm{\ta}{H^{r+4}}^2
\end{align*}
and
\begin{align*}
\Pi_{3,1}^{(3)}+\Pi_{3,1}^{(4)}+\Pi_{3,1}^{(5)}
\le&C\norm{\nabla \b}{H^{r+3}}^2+C\norm{\nabla \ta}{H^{r+3}}^2.
\end{align*}
As in the derivation of \eqref{yiming13}, we have
\begin{align*}
\Pi_{3,1}^{(6)}=&\big\langle{\Lambda^s}({\n}\cdot\nabla(\n\ta)),{\Lambda^s}(f_2)\big\rangle\nn\\
\le&\, \varepsilon\norm{ a}{H^{r+4}}^2+ C\delta^2\norm{ (\ta,\b)}{H^{r+4}}^2+
C\delta^2 \norm{{\n}\cdot\nabla \u}{H^{r+3}}^2.
\end{align*}
Recalling that $f_4 =-\u\cdot\nabla\b+\b\cdot\nabla\u-\b\div\u$, the last term
in (\ref{shiping4}) can be written as
\begin{align*}
\Pi_{4}=&{\big\langle{\Lambda^s} (f_4),{\Lambda^s}({\n}\cdot\nabla
\u)\big\rangle}\nn\\
=&-{\big\langle{\Lambda^s} (\u\cdot\nabla\b),{\Lambda^s}({\n}\cdot\nabla \u)\big\rangle}
+{\big\langle{\Lambda^s} (\b\cdot\nabla\u-\b\div\u),{\Lambda^s}({\n}\cdot\nabla \u)\big\rangle}
\nn\\
=&\Pi_{4,1}+\Pi_{4,2}.
\end{align*}
By H\"{o}lder's inequality
\begin{align*}
\Pi_{4,1}\le& \,C\norm{{\Lambda^s}(\u\cdot\nabla
\b)}{L^2}\norm{{\Lambda^s}({\n}\cdot\nabla \u)}{L^2}\nn\\
\le& \,C(\norm{\u}{L^\infty}\norm{\nabla \b}{H^s}+\norm{\nabla
\b}{L^\infty}\norm{ \u}{H^s} )\norm{{\Lambda^s}({\n}\cdot\nabla \u)}{L^2}\nn\\
\le&\frac{1}{16}\norm{{\Lambda^s}({\n}\cdot\nabla \u)}{L^2}^2+C(\norm{\u}{H^2}^2\norm{\nabla \b}{H^s}^2+\norm{\nabla \b}{H^2}^2\norm{ \u}{H^s}^2 )\nn\\
\le&\frac{1}{16}\norm{{\Lambda^s}({\n}\cdot\nabla \u)}{L^2}^2+C\norm{\u}{H^N}^2\norm{ \b}{H^{s+1}}^2\nn\\
\le&\frac{1}{16}\norm{{\Lambda^s}({\n}\cdot\nabla \u)}{L^2}^2+C\delta^2\norm{ \b}{H^{r+4}}^2,
\end{align*}
and
\begin{align*}
\Pi_{4,2}\le& C\norm{{\Lambda^s}(\b\cdot\nabla \u-\b\div
\u)}{L^2}\norm{{\Lambda^s}({\n}\cdot\nabla \u)}{L^2}\nn\\
\le& C(\norm{\b}{L^\infty}\norm{\nabla \u}{H^s}+\norm{\nabla \u}{L^\infty}\norm{ \b}{H^s} )\norm{{\Lambda^s}({\n}\cdot\nabla \u)}{L^2}\nn\\
\le& C(\norm{\b}{H^2}\norm{ \u}{H^{s+1}}+\norm{\nabla \u}{H^2}\norm{ \b}{H^s} )\norm{{\Lambda^s}({\n}\cdot\nabla \u)}{L^2}\nn\\
\le&\frac{1}{16}\norm{{\Lambda^s}({\n}\cdot\nabla \u)}{L^2}^2+C\norm{\u}{H^N}^2(\norm{ \b}{H^{s}}^2+\norm{ \b}{H^{2}}^2)\nn\\
\le&\frac{1}{16}\norm{{\Lambda^s}({\n}\cdot\nabla \u)}{L^2}^2+C\delta^2\norm{ \b}{H^{r+3}}^2.
\end{align*}
Inserting the bounds above in  (\ref{shiping4}) yields the desired inequality.
This completes the proof of Proposition \ref{shiping1}.
\end{proof}

\vskip .3in
\section{The dissipation of $\div\u$}
\label{sec:dudiss}

This section exploits the dissipative effect of $\div\u$. We explore the
interaction between $\div\u$ and $\theta$. The linearized system of $\div\u$
and $\theta$ is given by
\begin{align*}
	&\partial_t \div\u+ R\Delta a+ R\Delta \ta
	=-\Delta ({\n}\cdot\b),\\
	&\partial_t\theta-\kappa\Delta\ta+ \div \u=0.
\end{align*}
For simplicity, we ignore $R\Delta a$ and $-\Delta ({\n}\cdot\b)$. It is very
easy to derive that
	\begin{align*}
	&\partial_{tt} \div\u - \kappa \Delta \p_t  \div\u  - R\Delta
	\div\u=0,\\
	&\partial_{tt}\theta-\kappa\Delta\p_t- R\Delta \ta =0
\end{align*}
$\div\u$ and $\theta$ satisfies the same wave equation. The wave structure
reveals the dissipative nature of  $\div\u$. Making use of this structure,
we can prove the following proposition.

\begin{proposition} \label{ww}
	Let ${N}\ge 4r+7$ with $r>2$. Assume the solution $(a(t), \u(t), \ta(t),
	\b(t))$ to \eqref{m2} satisfies \eqref{shiping2}. Then
\begin{align*}
	&\norm{\div\u}{H^{{r+3}}}^2+\sum_{0\le s\le r+3}\frac{d}{dt}{\big\langle{\Lambda^{s}} \ta,{\Lambda^{s}}\div\u\big\rangle}
	\nn\\
	&\quad\le(\varepsilon+\delta^2)\norm{ \nabla a}{H^{r+3}}^2+C\norm{ \ta}{H^{r+5}}^2+C\delta^2\norm{ (\ta,\b)}{H^{r+4}}^2+ C\delta^2 \norm{{\n}\cdot\nabla \u}{H^{r+3}}^2
\end{align*}
where $\varepsilon>0$ is a fixed small number and the constant $C$ depends on
$\varepsilon>0$.
\end{proposition}

\begin{proof}
It follows from the equation
\begin{align*}
&\partial_t\theta-\Delta\ta+ \div \u=f_3
\end{align*}
that
\begin{align*}
\norm{{\Lambda^{s}}(\div\u)}{L^2}^2
=&-{\big\langle{\Lambda^{s}}\partial_t \ta,{\Lambda^{s}}\div\u\big\rangle}+{\big\langle{\Lambda^{s}} \Delta\ta,{\Lambda^{s}}\div\u\big\rangle}+{\big\langle{\Lambda^{s}} f_3,{\Lambda^{s}}\div\u\big\rangle}
\nn\\
:=&K_{1}+K_{2}+K_{3}.
\end{align*}
To estimate $K_1$, we rewrite it as
\begin{align*}
K_{1}=&-{\big\langle{\Lambda^{s}}\partial_t \ta,{\Lambda^{s}}\div\u\big\rangle}\nn\\
=&-\frac{d}{dt}{\big\langle{\Lambda^{s}}
\ta,{\Lambda^{s}}\div\u\big\rangle}+{\big\langle{\Lambda^{s}}\ta,{\Lambda^{s}}
\partial_t \div\u\big\rangle}\nn\\
:=&K_{1,1}+{\big\langle{\Lambda^{s}}\ta,{\Lambda^{s}} \partial_t
\div\u\big\rangle}.
\end{align*}
According to the equation of $\u$,
\begin{align*}
\partial_t\div\u=-\Delta\ta-\Delta a-\Delta(\n\cdot\b)+\div f_2.
\end{align*}
Therefore,
\begin{align*}
{\big\langle{\Lambda^{s}}\ta,{\Lambda^{s}} \partial_t \div\u\big\rangle}
=&-{\big\langle{\Lambda^{s}}\ta,{\Lambda^{s}} \Delta\ta\big\rangle}
-{\big\langle{\Lambda^{s}}\ta,{\Lambda^{s}} \Delta a\big\rangle}\nn\\
&-{\big\langle{\Lambda^{s}}\ta,{\Lambda^{s}} \Delta(\n\cdot\b)\big\rangle}
-{\big\langle{\Lambda^{s}}\ta,{\Lambda^{s}}\div f_2\big\rangle}\nn\\
:=&K_{1,2}+K_{1,3}+K_{1,4}+K_{1,5}.
\end{align*}
By H\"{o}lder's inequality,
\begin{align*}
K_{1,2}+K_{1,3}+K_{1,4}\,
\le&\varepsilon\norm{\nabla a}{H^{r+3}}^2+C\norm{\nabla
\ta}{H^{r+3}}^2+C\norm{\nabla \b}{H^{r+3}}^2, \\
K_{1,5}
=&
-{\big\langle{\Lambda^{s}}\ta,{\Lambda^{s}}\div f_2\big\rangle}\le
\varepsilon\norm{\nabla \ta}{H^{r+3}}^2+C\norm{{\Lambda^{s}}
f_2}{L^{2}}^2.
\end{align*}
Recall that
\begin{align*}
f_2{:=}&-\u\cdot\nabla \u+\b\cdot\nabla\b+\b\nabla\b+I(a)\nabla a-\theta\nabla
J(a)
\nn\\
&-I(a)({\n}\cdot\nabla \b+\b\cdot\nabla\b-{\n}\nabla \b-\b\nabla\b).
\end{align*}
It is not difficult to check that
\begin{align*}
\norm{\la^{s}  f_2}{L^{2}}^2\le& C\delta^2\norm{ (a,\ta,\b)}{H^{r+4}}^2+ C\delta^2 \norm{{\n}\cdot\nabla \u}{H^{r+3}}^2.
\end{align*}
That is,
\begin{align*}
K_{1,5}
\le&\varepsilon\norm{\nabla \ta}{H^{r+3}}^2+C\delta^2\norm{ (a,\ta,\b)}{H^{r+4}}^2+ C\delta^2 \norm{{\n}\cdot\nabla \u}{H^{r+3}}^2.
\end{align*}
For $0\le s\le r+3$,
\begin{align*}
K_{2}\le& C\norm{{\Lambda^s}\Delta\ta}{L^2}\norm{{\Lambda^s}\div\u}{L^2} \le
\frac{1}{16}\norm{\div\u}{H^{{r+3}}}^2+C\norm{\nabla \ta}{H^{r+4}}^2.
\end{align*}
To bound $K_3$, we recall that
\begin{align*}
f_3\stackrel{\mathrm{def}}{=}- \div(\theta \u)-\kappa
I(a)\Delta\ta+\frac{|\nabla\times\b|^2}{1+a}.
\end{align*}
Therefore,
\begin{align*}
K_{3}=&{\big\langle{\Lambda^{s}} f_3,{\Lambda^{s}}\div\u\big\rangle}\nn\\
\le&\frac{1}{16}
\|\div\u\|_{H^{r+3}}^2 +C\norm{\u}{H^{N}}^2\norm{\ta}{H^{r+4}}^2\nn\\
&+C\norm{a}{H^{N}}^2 \norm{\ta}{H^{r+5}}^2
+C(1+\norm{a}{H^{N}}^2)\norm{\b}{H^{N}}^2\norm{\b}{H^{r+4}}^2\nn\\
\le&\frac{1}{16} \|\div\u\|_{H^{r+3}}^2
+C\delta^2\norm{\ta}{H^{r+5}}^2+C(1+\delta^2)\delta^2\norm{\ta}{H^{r+4}}^2.
\end{align*}
This completes the proof of Proposition \ref{ww}.
\end{proof}

\vskip .3in
\section{The proof of Theorem \ref{dingli}}
\label{sec:pro}

This section completes the proof of Theorem \ref{dingli}.

\begin{proof}[Proof of Theorem \ref{dingli}] The framework of the proof is the
bootstrapping argument. First of all, the MHD system in \eqref{m2} with any
initial data $(a_0, \u_0, \ta_0, \b_0) \in H^N$ has a unique local solution.
This follows from a standard contraction mapping argument (see, e.g.,
\cite{MM}). The
bootstrapping argument is employed to prove the global existence and stability.
It starts with the ansatz that the solution $(a(t), \u(t), \ta(t),
\b(t))$ to \eqref{m2} satisfies
\begin{align*}
	\sup_{t\in[0,T]}\norm{(a(t), \u(t), \ta(t), \b(t))}{H^{N}}\le \delta
\end{align*}
for some $0<\delta<1$. We then show that
\begin{align}
	\sup_{t\in[0,T]}\norm{(a(t), \u(t), \ta(t), \b(t))}{H^{N}}\le
	\frac{\delta}{2}. \label{bb}
\end{align}	
We collect the estimates obtained in the previous sections:	
\begin{align}\label{xueying1}
&\frac12\frac{d}{dt}\left(\norm{(a,\u,\ta,\b)}{H^{\ell}}^2
+\int_{\T^3}\frac{\ta+a^2}{(1+a)^2}(\la^{\ell}a)^2 \,dx
\right)
+\kappa\norm{\nabla\ta}{H^{{\ell}}}^2
	+\sigma\norm{\nabla\b}{H^{{\ell}}}^2\nn\\
&\qquad\le C
Y_{\infty}(t)\norm{(a,\u,\ta,\b)}{H^{\ell}}^2,
\end{align}
\begin{align}\label{xueying2}
	&\norm{\nabla a}{H^{{r+3}}}^2+\sum_{0\le s\le r+3}\frac{d}{dt}{\big\langle{\Lambda^{s}} \u,{\Lambda^{s}}\nabla a\big\rangle}\nn\\
	&\quad\le C\norm{\div\u}{H^{{r+3}}}^2+
C(1+\delta^2)\norm{ (\ta,\b)}{H^{r+4}}^2+ C\delta^2 \norm{{\n}\cdot\nabla
\u}{H^{r+3}}^2
\end{align}
and
\begin{align}\label{xueying20}
	&\norm{\div\u}{H^{{r+3}}}^2+\sum_{0\le s\le r+3}\frac{d}{dt}{\big\langle{\Lambda^{s}} \ta,{\Lambda^{s}}\div\u\big\rangle}
	\nn\\
	&\quad\le(\varepsilon+\delta^2)\norm{ \nabla a}{H^{r+3}}^2+C\norm{
	\ta}{H^{r+5}}^2+C\delta^2\norm{ (\ta,\b)}{H^{r+4}}^2+ C\delta^2
	\norm{{\n}\cdot\nabla \u}{H^{r+3}}^2.
\end{align}
Adding \eqref{xueying2} and \eqref{xueying20}, and choosing
$\varepsilon,\delta$
small enough, we have
\begin{align}\label{xueying3}
	&\norm{\nabla a}{H^{{r+3}}}^2+\norm{\div\u}{H^{{r+3}}}^2+\sum_{0\le s\le r+3}\frac{d}{dt}\left({\big\langle{\Lambda^{s}} \u,{\Lambda^{s}}\nabla a\big\rangle}+{\big\langle{\Lambda^{s}} \ta,{\Lambda^{s}}\div\u\big\rangle}\right)\nn\\
	&\quad\le C\norm{ (\ta,\b)}{H^{r+5}}^2+ C\delta^2 \norm{{\n}\cdot\nabla \u}{H^{r+3}}^2.
\end{align}
Summing up \eqref{xueying3} and \eqref{shiping3} and choosing $\delta$small
enough, we obtain
\begin{align}\label{xueying5}
	&\norm{\nabla a}{H^{{r+3}}}^2+\norm{\div\u}{H^{{r+3}}}^2
+\norm{{\n}\cdot\nabla \u}{H^{r+3}}^2\nn\\
&\qquad+\sum_{0\le s\le r+3}\frac{d}{dt}\left({\big\langle{\Lambda^{s}} \u,{\Lambda^{s}}\nabla a\big\rangle}+{\big\langle{\Lambda^{s}} \ta,{\Lambda^{s}}\div\u\big\rangle}\right)-\sum_{0\le s\le r+3}\frac{d}{dt}\big\langle{\Lambda^s}\b,{\Lambda^s}({\n}\cdot\nabla \u)\big\rangle\nn\\
	&\quad\le C\norm{ (\ta,\b)}{H^{r+5}}^2.
\end{align}
Taking $\ell=r+4$ in \eqref{xueying1} yields
\begin{align}\label{xueying6}
&\frac12\frac{d}{dt}\left(\norm{(a,\u,\ta,\b)}{H^{r+4}}^2
+\int_{\T^3}\frac{\ta+a^2}{(1+a)^2}(\la^{r+4}a)^2 \,dx
\right)
+\kappa\norm{\nabla\ta}{H^{{r+4}}}^2
	+\sigma\norm{\nabla\b}{H^{{r+4}}}^2\nn\\
&\qquad\le C
Y_{\infty}(t)\norm{(a,\u,\ta,\b)}{H^{r+4}}^2.
\end{align}
Multiplying \eqref{xueying6} by a suitable large constant $\gamma$ and
adding to \eqref{xueying5} give rise to
\begin{align}\label{xueying7}
&\frac12\frac{d}{dt}\Bigg(\gamma\norm{(a,\u,\ta,\b)}{H^{r+4}}^2
+\gamma\int_{\T^3}\frac{\ta+a^2}{(1+a)^2}(\la^{s}a)^2 \,dx
\nn\\
&\qquad+\sum_{0\le s\le r+3}\left({\big\langle{\Lambda^{s}} \u,{\Lambda^{s}}\nabla a\big\rangle}+{\big\langle{\Lambda^{s}} \ta,{\Lambda^{s}}\div\u\big\rangle}-\big\langle{\Lambda^s}\b,{\Lambda^s}({\n}\cdot\nabla \u)\big\rangle\right)\Bigg)\nn\\
&\qquad+\gamma\kappa\norm{\nabla\ta}{H^{{r+4}}}^2
	+\gamma\sigma\norm{\nabla\b}{H^{{r+4}}}^2+\norm{\nabla a}{H^{{r+3}}}^2+\norm{\div\u}{H^{{r+3}}}^2
+\norm{{\n}\cdot\nabla \u}{H^{r+3}}^2\nn\\
&\quad\le C\gamma
Y_{\infty}(t)\norm{(a,\u,\ta,\b)}{H^{r+4}}^2.
\end{align}
Recall the definition of $Y_{\infty}$ in (\ref{ed3+1}),
\begin{align*}
	Y_{\infty}(t)\stackrel{\mathrm{def}}{=}&
	\|(a,\u,\ta,\b)\|_{L^{\infty}}+(1+\|a\|_{L^{\infty}}^2)\|(a,\u,\ta,\b)\|_{L^{\infty}}^2
	+(1+
	\norm{a}{L^\infty})\norm{(\nabla a,\nabla\u,\nabla
	\ta,\nabla\b)}{L^\infty}\nn\\
	&+\norm{\Delta
	\ta}{L^\infty}+(1+\|(a,\u,\b)\|_{L^{\infty}}^2+\|\nabla\u\|_{L^{\infty}}^2)\norm{(\nabla
	 a,\nabla\u,\nabla \ta,\nabla\b)}{L^\infty}^2.
\end{align*}
That is, $Y_{\infty}$ essentially contains $\|(a,\u,\ta,\b)\|_{L^{\infty}}$,
$\norm{(\nabla a,\nabla\u,\nabla
	\ta,\nabla\b)}{L^\infty}$, $\norm{\Delta \ta}{L^\infty}$ or their squares.
	Without loss of generality, we estimate some of them.
	The other terms can be bounded similarly.
For any ${{N}}\ge 2r+5$, according to \eqref{mor},
\begin{align*}
\norm{ \u}{H^{3}}^2\le& \,C\norm{{\n}\cdot\nabla \u}{H^{r+3}}^2,\quad
\norm{ \u}{H^{r+4}}^2\le\norm{ \u}{H^{3}}\norm{ \u}{H^{{{N}}}}\le \,
C\delta\norm{{\n}\cdot\nabla \u}{H^{r+3}}.
\end{align*}
Therefore,
\begin{align}\label{xueying9}
\gamma\norm{\nabla \b}{L^\infty}\norm{\u}{H^{r+4}}^2
\le& C\gamma\norm{\nabla \b}{H^{r+3}}\norm{\u}{H^{r+4}}^2\nn\\
\le& \frac {\gamma\sigma}{2}\norm{ \nabla\b}{H^{r+3}}^2+C\norm{\u}{H^{r+4}}^4\nn\\
\le& \frac {\gamma\sigma}{2}\norm{\nabla
\b}{H^{r+3}}^2+C\delta^2\norm{{\n}\cdot\nabla \u}{H^{r+3}}^2,
\end{align}
\begin{align}\label{xueying10}
\gamma\norm{\nabla \u}{L^\infty}\norm{\u}{H^{r+4}}^2
\le& \, C \gamma\norm{\nabla \u}{H^{2}}\norm{\u}{H^{r+4}}^2\nn\\
\le& \,C  \gamma\norm{{\n}\cdot\nabla \u}{H^{r+3}}
\delta\,\norm{{\n}\cdot\nabla \u}{H^{r+3}}
\nn\\
\le& \,C\delta\norm{{\n}\cdot\nabla \u}{H^{r+3}}^2,
\end{align}
\begin{align}\label{xueying11}
\gamma\norm{\nabla a}{L^\infty}\norm{\u}{H^{r+4}}^2
\le& C\gamma\norm{\nabla a}{H^{r+3}}\norm{\u}{H^{r+4}}^2\nn\\
\le& \frac14 \norm{\nabla a}{H^{r+3}} + C\gamma^2 \,\norm{\u}{H^{r+4}}^4\nn\\
\le&  \frac14 \norm{\nabla a}{H^{r+3}} +\,C\gamma^2
\,\delta^2\norm{{\n}\cdot\nabla
\u}{H^{r+3}}^2
\end{align}
and
\begin{align}\label{xueying110}
	\gamma\norm{\Delta\ta}{L^\infty}\norm{\u}{H^{r+4}}^2
	\le& C\gamma\norm{\nabla a}{H^{r+3}}\norm{\u}{H^{r+4}}^2\nn\\
	\le& \frac14 \norm{\nabla a}{H^{r+3}} + C\gamma^2
	\,\norm{\u}{H^{r+4}}^4\nn\\
	\le&  \frac14 \norm{\nabla a}{H^{r+3}} +\,C\gamma^2
	\,\delta^2\norm{{\n}\cdot\nabla
		\u}{H^{r+3}}^2.
\end{align}
For notational convenience, we set
\begin{align*}
{\mathcal{E}(t)}=&\gamma\norm{(a,\u,\ta,\b)}{H^{r+4}}^2
+\gamma\int_{\T^3}\frac{\ta+a^2}{(1+a)^2}(\la^{s}a)^2 \,dx
\nn\\
&\qquad+\sum_{0\le s\le r+3}\left({\big\langle{\Lambda^{s}} \u,{\Lambda^{s}}\nabla a\big\rangle}+{\big\langle{\Lambda^{s}} \ta,{\Lambda^{s}}\div\u\big\rangle}-\big\langle{\Lambda^s}\b,{\Lambda^s}({\n}\cdot\nabla \u)\big\rangle\right)
\end{align*}
and
\begin{align*}
{\mathcal{D}(t)}=&\gamma\kappa\norm{\nabla\ta}{H^{{r+4}}}^2
	+\gamma\sigma\norm{\nabla\b}{H^{{r+4}}}^2+\norm{\nabla a}{H^{{r+3}}}^2+\norm{\div\u}{H^{{r+3}}}^2
+\norm{{\n}\cdot\nabla \u}{H^{r+3}}^2.
\end{align*}
By choosing $\delta>0$ sufficiently small and inserting the upper bounds in
(\ref{xueying9}), (\ref{xueying10}), (\ref{xueying11}) and (\ref{xueying110})
in
(\ref{xueying7}), we obtain
\begin{align}\label{ming13}
\frac{d}{dt}{\mathcal{E}(t)}+\frac12{\mathcal{D}(t)}\le 0.
\end{align}
For any ${N}\ge 4r+7$, by the interpolation inequality, we have
\begin{align*}
\norm{ \u}{H^{r+4}}^2\le&\norm{ \u}{H^{3}}^{\frac32}\norm{
\u}{H^{{N}}}^{\frac12}\le C\delta^{\frac12}\norm{{\n}\cdot\nabla
\u}{H^{r+3}}^{\frac32}.
\end{align*}
Thanks to Lemma \ref{poincareineq},
\begin{align}\label{}
	\norm{\ta}{H^{r+4}}^2
	\approx& \|{\ta}\|_{L^{2}}+\|{\ta}\|_{\dot{H}^{r+4}}^2\nonumber\\
	\le&
	C\|\nabla\ta\|_{L^2}^2+\|\nabla\u\|_{L^2}^4+\|\nabla\b\|_{L^2}^4+\norm{\nabla\ta}{H^{r+4}}^2\nonumber\\
	\le&\norm{\nabla\ta}{H^{r+4}}^2+\delta^2\norm{\u}{H^{r+4}}^2+\delta^2\norm{\b}{H^{r+4}}^2.
\end{align}
Therefore,
\begin{align}
	{\mathcal{E}(t)}
	\le& \, C(\norm{ a}{H^{r+4}}^2 +
	\norm{(\ta,\b)}{H^{r+4}}^2+\norm{\u}{H^{r+4}}^2)\nn\\
	\le& \, C(\norm{ a}{H^{r+4}}^2 +
	\norm{\nabla\ta}{H^{r+4}}^2+
	\norm{\b}{H^{r+4}}^2+\norm{\u}{H^{r+4}}^2)\nn\\
	\le& \,C\norm{ a}{H^{r+4}}^{\frac32}\norm{
		a}{H^{{r+4}}}^{\frac12}+C\norm{ \nabla\ta}{H^{r+4}}^{\frac32}\norm{
		\nabla\ta}{H^{{r+4}}}^{\frac12}+C\norm{\b}{H^{3}}^{\frac32}\norm{
		\b}{H^{{N}}}^{\frac12}+C\norm{ \u}{H^{3}}^{\frac32}\norm{
		\u}{H^{{N}}}^{\frac12}\nn\\
	\le&  \,C\delta^{\frac12}\norm{\nabla
		a}{H^{r+4}}^{\frac32}+C\delta^{\frac12}\norm{\nabla
		\ta}{H^{r+4}}^{\frac32}+
	C\delta^{\frac12}\norm{\nabla
		\b}{H^{r+4}}^{\frac32}+C\delta^{\frac12}\norm{{\n}\cdot\nabla
		\u}{H^{r+3}}^{\frac32}\nn\\
	\le& \,(\mathcal{D}(t))^{\frac34}, \label{ebd}
\end{align}
where we have used Poincare's inequality on $a$ thanks to the fact that $a$ has
mean-zero. Inserting \eqref{ebd} in \eqref{ming13} yields
a Laputa-type inequality,
\begin{align*}
\frac{d}{dt}{\mathcal{E}(t)}+c({\mathcal{E}(t)})^{\frac43}\le 0,
\end{align*}
which implies
$$
{\mathcal{E}(t)}\le C(1+t)^{-3}.
$$
It is easily seen that
$$
{\mathcal{E}(t)}\ge\norm{(a,\u,\ta,\b)}{H^{{r+4}}}^2
$$
and thus
\begin{align}\label{ming17}
\norm{(a,\u,\ta,\b)}{H^{{r+4}}}\le C(1+t)^{-\frac32}.
\end{align}
Taking $\ell={N}$ in \eqref{xueying7} and using the embedding relation, we find
\begin{align*}
&\frac{d}{dt}\norm{(a,\u,\ta,\b)}{H^{N}}^2+\sigma\norm{\nabla\b}{H^{N}}^2
+\kappa\norm{\nabla\ta}{H^{N}}^2
\le C Z(t)\norm{(a,\u,\ta,\b)}{H^{N}}^2
\end{align*}
with
\begin{align*}
Z(t)\stackrel{\mathrm{def}}{=}&
\|(a,\u,\ta,\b)\|_{H^2}+(1+\|a\|_{H^2}^2)\|(a,\u,\ta,\b)\|_{H^2}^2
+(1+
\norm{a}{H^2})\norm{( a,\u, \ta,\b)}{H^3}\nn\\
&+\norm{ \ta}{H^4}+(1+\|(a,\u,\b)\|_{H^2}^2+\|\u\|_{H^3}^2)\norm{( a,\u, \ta,\b)}{H^3}^2.
\end{align*}
Thanks to \eqref{ming17},
\begin{align*}
\int_0^tZ(\tau)\,d\tau\le C.
\end{align*}
It then follows from Gr\"{o}nwall's inequality that
\begin{align*}
\norm{(a,\u,\ta,\b)}{H^{N}}^2
\le&C\norm{(a_0,\u_0,\ta_0,\b_0)}{H^{N}}^2
\le C\varepsilon^2.
\end{align*}
We finish the proof of \eqref{bb} by taking $\varepsilon$ small enough so that
$C\varepsilon\le \delta/2$. This finishes the bootstrapping argument and thus
the proof of Theorem \ref{dingli} is complete.
\end{proof}

\section{Declarations}

\bigskip
\noindent{\bf Ethical Approval: }\

 We certify that this manuscript is original and has not been published and will not be submitted elsewhere for publication while being considered by Annals of PDE.

\bigskip
\noindent{\bf Conflict of Interest:} \

The authors declare that they have no conflict of interest.

\bigskip
\noindent{\bf Data availability statement:}\

 Data sharing not applicable to this article as no datasets were generated or analysed during the
current study.

\bigskip
\noindent{\bf Authors' contributions:   }\

Jiahong Wu, Fuyi  Xu and Xiaoping Zhai contributed equally to this work.

\bigskip
\noindent{\bf Corresponding author:   }\ Xiaoping Zhai.

\bigskip
\noindent{ \bf Funding Statement:  }\
Wu was partially supported by the National Science Foundation of the United
States under DMS 2104682 and DMS 2309748. Xu was partially supported by   the
National Natural Science Foundation of China 12326430,
and the  Natural Science Foundation of Shandong Province ZR2021MA017. Zhai was partially supported by the Guangdong Provincial Natural Science
Foundation under grant  2024A1515030115.

\end{document}